\documentclass[12pt,final]{amsart}
\usepackage[utf8]{inputenc}
\usepackage[T1]{fontenc}
\usepackage[letterpaper,centering]{geometry}
\usepackage{lmodern}
\usepackage{mathrsfs}
\usepackage{amsmath,amssymb,amsfonts}
\usepackage[unicode,pdfborder={0 0 0}]{hyperref}
\usepackage[ps,dvips,arrow,matrix,tips,line]{xy}
{\setbox0\hbox{$ $}}\fontdimen16\textfont2=\fontdimen17\textfont2
\entrymodifiers={+!!<0pt,\the\fontdimen22\textfont2>}
\SelectTips{cm}{11}

\input cyracc.def
\DeclareFontFamily{U}{russian}{}
\DeclareFontShape{U}{russian}{m}{n}
        { <5><6> wncyr5
        <7><8><9> wncyr7
        <10><10.95><12><14.4><17.28><20.74><24.88> wncyr10 }{}
\DeclareSymbolFont{Russian}{U}{russian}{m}{n}
\DeclareSymbolFontAlphabet{\mathcyr}{Russian}
\makeatletter
\let\@math@cyr\mathcyr
\renewcommand{\mathcyr}[1]{\@math@cyr{\cyracc #1}}
\makeatother

\theoremstyle{plain}
\newtheorem{thm}{Theorem}[section]

\newtheorem{conj}[thm]{Conjecture}
\newtheorem{question}[thm]{Question}
\newtheorem{cor}[thm]{Corollary}
\newtheorem{prop}[thm]{Proposition}
\theoremstyle{definition}
\newtheorem{defn}[thm]{Definition}
\newtheorem{rmk}[thm]{Remark}
\newtheorem{rmks}[thm]{Remarks}
\newtheorem{example}[thm]{Example}

\numberwithin{equation}{section}

\newcommand{\congru}[3]{#1 \equiv #2 \!\!\mod #3}
\newcommand{\poonen}{{\smash{\bullet}}}
\newcommand{\Gm}{\mathbf{G}_\mathrm{m}}

\makeatletter
\def\myrightarrow{{\setbox\z@\hbox{$\rightarrow$}\dimen0\ht\z@\multiply\dimen0 6\divide\dimen0 10\ht\z@\dimen0\box\z@}}
\def\myrightarrowfill@{\arrowfill@\relbar\relbar\myrightarrow}
\newcommand{\myxrightarrow}[2][]{\ext@arrow 0359\myrightarrowfill@{#1}{#2}}
\def\myleftarrow{{\setbox\z@\hbox{$\leftarrow$}\dimen0\ht\z@\multiply\dimen0 6\divide\dimen0 10\ht\z@\dimen0\box\z@}}
\def\myleftarrowfill@{\arrowfill@\myleftarrow\relbar\relbar}
\newcommand{\myxleftarrow}[2][]{\ext@arrow 3095\myleftarrowfill@{#1}{#2}}
\makeatother

\newcommand{\sF}{{\mathscr F}}
\newcommand{\sO}{{\mathscr O}}

\newcommand{\sX}{{\mathscr X}}

\newcommand{\A}{{\mathbf A}}

\renewcommand{\C}{{\mathbf C}}

\renewcommand{\Im}{{\mathrm{Im}}}

\newcommand{\N}{{\mathbf N}}
\renewcommand{\P}{{\mathbf P}}
\newcommand{\Q}{{\mathbf Q}}
\newcommand{\R}{{\mathbf R}}

\newcommand{\Z}{{\mathbf Z}}

\newcommand{\bark}{\mkern2mu\overline{\mkern-2mu \smash{k}\vphantom{Y}\mkern-1.5mu}\mkern1.5mu}
\newcommand{\barY}{\mkern3.5mu\overline{\mkern-3.5mu Y\mkern-3.8mu}\mkern3.8mu}
\newcommand{\barD}{\mkern5.5mu\overline{\mkern-5.5mu D\mkern-3.4mu}\mkern3.4mu}
\newcommand{\pibar}{\mkern2.2mu\overline{\mkern-2.2mu \pi\mkern-2mu}\mkern2mu}
\newcommand{\CH}{\mathrm{CH}}
\newcommand{\Gal}{\mathrm{Gal}}
\newcommand{\Sym}{\mathrm{Sym}}
\newcommand{\Aut}{\mathrm{Aut}}

\newcommand{\inv}{\mathrm{inv}}

\newcommand{\Pic}{\mathrm{Pic}}

\newcommand{\End}{\mathrm{End}}
\newcommand{\Br}{\mathrm{Br}}
\newcommand{\Spec}{\mathrm{Spec}}
\newcommand{\Div}{\mathrm{Div}}
\renewcommand{\phi}{\varphi}
\renewcommand{\emptyset}{\varnothing}

\newcommand{\Ker}{{\mathrm{Ker}}}

\newcommand{\Hom}{{\mathrm{Hom}}}

\newcommand{\Homrond}{\mathscr{H}\mkern-4muom}

\newcommand{\et}{\text{ét}}

\date{April 26th, 2016; revised on August 25th, 2016}
\title[Rational points and zero-cycles on rationally connected varieties]{Rational points and zero-cycles on rationally connected varieties over number fields}

\author{Olivier Wittenberg}
\address{D\'epartement de math\'ematiques et applications, \'Ecole normale sup\'erieure, 45~rue d'Ulm, 75230 Paris Cedex 05, France}
\email{wittenberg@dma.ens.fr}

\begin{document}

\begin{abstract}
We report on progress in the qualitative study of rational points on
rationally connected varieties over number fields,
also examining integral points, zero-cycles, and non-rationally connected varieties.
One of the main objectives is to
 highlight
and explain the many recent interactions with analytic number theory.
\end{abstract}

\maketitle

\section{Introduction}

A striking aspect of recent progress in the qualitative study
of rational points on rationally connected varieties over number fields
has been the emergence of new interactions between analytic number theory,
on the one hand,
and algebro-geometric methods, on the other hand.
At least three other related directions have witnessed significant advances
in the past decade:
the study of various obstructions to the existence of rational points on non-rationally connected varieties,
now all
encapsulated in the so-called étale Brauer--Manin obstruction, and of
their insufficiency to explain the lack of rational points on certain types of varieties;
the study of integral points on non-proper varieties, reshaped by the unexpected
discovery that
Brauer--Manin obstructions are relevant to integral points too;
and the existence and approximation properties of zero-cycles on varieties defined over number fields, which were studied
concurrently with similar properties for rational points, progress being achieved in parallel in the two frameworks although
results on zero-cycles do not depend on analytic number theory.

These expository notes are an attempt
to give a coherent and up-to-date overview
of the above topics and of the
general context to which they belong.
The first half of the text, \textsection\ref{sec:definitionsandquestions},
formulates and discusses
some of the main qualitative questions that can be asked about rational points,
integral points, and zero-cycles, on algebraic varieties defined over number fields.
To attack these questions for varieties which are either rationally connected or close enough to being so,
a number of general tools are available:
methods from analytic number theory,
Galois cohomological methods,
and two geometric methods known as
the descent method
and the fibration method.
The second half of the text, \textsection\ref{sec:methods}, is devoted to these methods,
and to the existing interactions between them, through examples.

Many interesting topics had to be left out altogether or are only briefly mentioned;
for instance, we do not touch on the quantitative aspects of rational and integral points
or on the analogues over function fields of complex curves
of the questions discussed here over number fields.
The reader will find complementary
material in~\cite{peyrebourbaki},
\cite{browningquantitativearithmetic},
\cite{tschinkelmanyrationalpoints},
\cite{abramovich},
\cite{hassettwa},
\cite{wittrc}.

\bigskip
\emph{Acknowledgements.}
I am grateful to the organisers of the 2015 AMS Algebraic Geometry Summer Institute
in Salt Lake City 
for their invitation to contribute to this volume.
I would also like to thank
Tim Browning, Jean-Louis Colliot-Thélène, Cyril Demarche, Yonatan Harpaz and Florent Jouve
for discussions on several points of these notes,
as well as the audiences in Rio de Janeiro and in Atlanta, where I~lectured
on these topics in~2015,
and the referees for their useful comments.

\section{Over number fields: general context}
\label{sec:definitionsandquestions}

We fix once and for all a number field~$k$.  We denote by~$\Omega$ the set of places of~$k$,
by $\Omega_f$ (resp., $\Omega_\infty$) the subset of finite (resp., infinite) places,
by~$k_v$ the completion of~$k$ at $v \in \Omega$, by $\sO_v \subset k_v$ the ring of integers of~$k_v$ for $v\in \Omega_f$,
and by~$\A_k$ the ring
of ad\`eles of~$k$, defined as $\A_k = \varinjlim \A_{k,S}$ where the direct limit ranges over the finite subsets~$S$ of~$\Omega$ ordered by inclusion
and where
$\A_{k,S} = \prod_{v \in S} k_v \times \prod_{v \in \Omega\setminus S} \sO_v$
for any finite subset $S \subset \Omega$
containing~$\Omega_\infty$.  The reader is welcome to assume throughout that $k=\Q$ and that
the fields~$k_v$ for $v \in\Omega$
are the fields~$\Q_p$ of $p$\nobreakdash-adic numbers for the various primes~$p$
and the field~$\R$ of real numbers.

Let~$X$ be a smooth, proper, geometrically irreducible variety over~$k$.
As~$X$ is proper, we have $X(\A_k)=\prod_{v \in \Omega}X(k_v)$
as topological spaces.
Any rational point gives rise to an adelic point through
the diagonal embedding $X(k) \subseteq X(\A_k)$.

It has long been understood that determining whether $X(\A_k)$ is empty or not
can be reduced to a finite computation (often easy in practice).
This follows from
the Lang--Weil estimates and
Hensel's lemma, which guarantee that $X(k_v) \neq\emptyset$ as soon as~$v$
does not belong to a certain computable finite set of places of~$k$;
from the work of Tarski and Seidenberg,
for the real places of~$k$; and from Hensel's lemma for the finitely many remaining finite places.

Sometimes,
the existence of an adelic point on~$X$ is sufficient to imply the existence of a rational point.
This is so, for instance, when~$X$ is a quadric, according to the Hasse--Minkowski theorem.
In general, however, further obstructions exist and must be taken into account.  We discuss some of these in
\textsection\textsection\ref{subsec:bmo}--\ref{subsec:furtherobstructions}.
In~\textsection\ref{subsec:requivalence}
we introduce $R$\nobreakdash-equivalence of rational points.
In the case of rationally connected varieties,
a precise conjecture describing the closure of~$X(k)$ in~$X(\A_k)$ is available (see~\textsection\ref{subsec:bmo});
we examine non-rationally connected varieties in~\textsection\ref{subsec:nonrc}.
We then consider variants of these questions in the context of
zero-cycles, and in that of integral points on non-proper varieties,
in~\textsection\ref{subsec:zerocycles} and~\textsection\ref{subsec:integralpoints}, respectively.

\subsection{Brauer--Manin obstruction}
\label{subsec:bmo}

For any scheme~$S$, let $\Br(S)=H^2_\et(S,\Gm)$ denote the (cohomological) Brauer group of~$S$.
Local and global class field theory describe this group when $S=\Spec(k_v)$
and $S=\Spec(k)$.  Specifically, for $v \in \Omega$, local class field theory equips $\Br(k_v)$ with a
canonical injection $\inv_v:\Br(k_v)\to \Q/\Z$ (an isomorphism when~$v$ is finite,
while $\Br(\R)=\Z/2\Z$ and $\Br(\C)=0$), and global class field theory provides a canonical exact sequence
\begin{align}
\label{eq:sebr}
\xymatrix@C=2.5em{
0 \ar[r] & \Br(k) \ar[r] & \displaystyle\bigoplus_{v\in\Omega}\Br(k_v) \ar[r]^(.58){\sum\inv_v}& \Q/\Z \ar[r] & 0 \rlap{\text{.}}
}
\end{align}
The assertion that~\eqref{eq:sebr} is a complex is the \emph{global reciprocity law}, a generalisation
of the quadratic reciprocity law and of many of its variants (cubic, biquadratic...).

\newcommand{\citemaninicm}{Manin~\cite{maninicm}}
\begin{defn}[\citemaninicm]
\label{def:bmset}
For a subgroup $B \subseteq \Br(X)$, we let
\begin{align*}
X(\A_k)^B = \left\{ (P_v)_{v \in \Omega} \in X(\A_k) \mkern3mu;\mkern3mu\forall \alpha\in B, \sum_{v\in\Omega}\inv_v\mkern2.5mu\alpha(P_v)=0 \in \Q/\Z \right\}\rlap{\text{,}}
\end{align*}
where $\alpha(P_v)\in \Br(k_v)$ denotes the pull-back of~$\alpha$ along $P_v \in X(k_v)$.
(To make sense of the infinite sum, one first checks that only finitely many of its terms are non-zero.)
The \emph{Brauer--Manin set} of~$X$ is the set $X(\A_k)^{\Br(X)}$.
\end{defn}

By the global reciprocity law, we have the following sequence of inclusions:
\begin{align}
\label{eq:seqinclusions}
X(k)\subseteq X(\A_k)^{\Br(X)} \subseteq X(\A_k)\rlap{\text{.}}
\end{align}

\newcommand{\citeiskovskikh}{\cite{iskocounterexample}, \cite{css80}}
\begin{example}[\citeiskovskikh]
\label{ex:isko}
Let~$X$ be a smooth and proper model of the affine surface over~$\Q$ defined by $x^2+y^2=(3-t^2)(t^2-2)$.
It is easy to see that $X(\A_\Q)\neq\emptyset$.  However,
letting~$\Q(X)$ denote the function field of~$X$,
one can check that
the quaternion algebra $(-1,3-t^2) \in \Br(\Q(X))$ extends (uniquely) to a class $\alpha \in \Br(X)$
and that $\alpha(P_v)=0$ for any $P_v \in X(\Q_v)$ and any $v \in \Omega\setminus \{2\}$ while
$\inv_2\mkern2.5mu\alpha(P_2)=\frac{1}{2}$ for any $P_2 \in X(\Q_2)$,
so that the Brauer--Manin set of~$X$ is empty and hence~$X$ has no rational point.
\end{example}

It is a general fact that $X(\A_k)^{\Br(X)}$ is
a closed subset of $X(\A_k)$ and therefore contains the closure of~$X(k)$.
Sometimes, the Brauer--Manin set is nevertheless too crude an approximation for the closure of~$X(k)$
for a trivial reason:
if~$v$ denotes an archimedean place,
the evaluation of classes of~$\Br(X)$ on~$X(k_v)$ is constant on the connected components
whereas the closure of~$X(k)$ in~$X(k_v)$ need not be a union of connected components.
(Example: the plane curve $x^3+y^3+z^3=0$ in~$\P^2_\Q$ only has three rational points.)
To circumvent this defect, it is convenient
to squash to single points the connected components
at the archimedean places and
to consider, with Poonen (see~\cite[\textsection2]{stoll}), the set of \emph{modified adelic points}
\begin{align}
X(\A_k)_\poonen=\prod_{v \in \Omega_f} X(k_v) \times \prod_{v \in \Omega_\infty} \pi_0(X(k_v))
\end{align}
(where~$\pi_0$ stands for the set of connected components)
and the \emph{modified Brauer--Manin set} $X(\A_k)^{\Br(X)}_\poonen=\Im\big(X(\A_k)^{\Br(X)}\to X(\A_k)_\poonen\big)$.

For arbitrary varieties,
rational points need not be dense even in the modified Brauer--Manin set;
see~\textsection\ref{subsec:etalecovers}
and~\textsection\ref{subsec:furtherobstructions} below.
For varieties whose geometry is simple enough, however, there is hope that the Brauer--Manin obstruction controls rational points entirely.
The following represents the main open question in the arithmetic of rationally connected varieties over number fields.
In the case
of geometrically rational surfaces, it was put forward by
Colliot-Thélène and Sansuc in the 1970's.

\newcommand{\citeconjct}{Colliot-Thélène~\cite[p.~174]{ctbudapest}}
\begin{conj}[\citeconjct]
\label{conj:ctrc}
Let~$X$ be a smooth, proper, geometrically irreducible variety defined over a number field~$k$.
If~$X$ is rationally connected, then~$X(k)$ is dense in $X(\A_k)^{\Br(X)}$.
\end{conj}

By ``$X$ is rationally connected'', we mean that
the variety $X\otimes_k \bark$ is rationally connected in the sense of Campana, Kollár, Miyaoka and Mori,
where~$\bark$ denotes an algebraic closure of~$k$.
Equivalently,
for any algebraically closed field~$K$
containing~$k$, two general $K$\nobreakdash-points of~$X$ can be joined by a rational curve defined over~$K$.
We recall that examples of rationally connected varieties include geometrically unirational varieties
(trivial),
Fano varieties (see~\cite{campanafano}, \cite{kmm}), and
fibrations into rationally connected varieties over a rationally
connected base (see~\cite{ghs}).

Conjecture~\ref{conj:ctrc} is wide open even for surfaces.

\begin{rmks}
\label{rmks:conjct}
(i) The group $\Br_0(X)=\Im(\Br(k)\to\Br(X))$ does not contribute to the Brauer--Manin set:
the global reciprocity law
implies the equality
 $X(\A_k)^B=X(\A_k)^{B+\Br_0(X)}$
for any subgroup $B \subseteq \Br(X)$.
When $\Br(X)/\Br_0(X)$ is finite, the subset $X(\A_k)^{\Br(X)} \subseteq X(\A_k)$ is therefore
cut out by finitely many conditions. In particular,
in this case,
it is an open subset of $X(\A_k)$.

(ii) When~$X$ is rationally connected, or, more generally, when~$X$ is
simply connected and $H^2(X,\sO_X)=0$, an analysis of the Hochschild--Serre
spectral sequence and of the Brauer group of $X\otimes_k\bark$ implies that
 $\Br(X)/\Br_0(X)$ is finite
(see \cite[Lemma~1.1]{ctskogoodreduction}).  Presumably, this quotient should
be finite under the sole assumption that~$X$ is simply connected, but this is out of reach of current knowledge.
(As follows from~\cite{ctskodescentebr} and
\cite[\textsection4]{ctskogoodreduction},
the finiteness of the $\ell$\nobreakdash-primary torsion subgroup of $\Br(X)/\Br_0(X)$ is equivalent,
when~$X$ is simply connected,
to the $\ell$\nobreakdash-adic Tate conjecture for divisors on~$X$.)

(iii) Letting~$v$ be any archimedean place of~$k$
and noting that
the image of the projection map $X(\A_k)^{\Br(X)} \to X(k_v)$ is a union of connected components,
we see that
Conjecture~\ref{conj:ctrc} would imply that
rational points are dense for the Zariski topology
on any (smooth) rationally connected variety
which possesses a rational point.
(By Remarks~\ref{rmks:conjct}~(i) and~(ii), this would also follow from the
weaker conjecture that~$X(k)$ is dense in $X(\A_k)^{\Br(X)}_\poonen$.)
Even Zariski density of rational points after a finite extension of the ground field
is unknown for conic bundles over a rational surface
(see~\cite{hassettpotentialdensity} for more on this topic).

(iv) The following are birational invariants of smooth and proper varieties over~$k$:
the Brauer group
(\cite[III, \textsection7]{grbr}), the property that $X(k)\neq\emptyset$
(Nishimura's lemma, see \cite[Prop.~A.6]{reichsteinyoussin}),
the property that $X(\A_k)^{\Br(X)}\neq\emptyset$
(see \cite[Prop.~6.1]{cps}),
and,
when $\Br(X)/\Br_0(X)$ is finite,
in particular when~$X$ is rationally connected,
also the
property that~$X(k)$ is dense in~$X(\A_k)^{\Br(X)}$
(\emph{loc.\ cit.}).

(v) If $X \subset \P^n_k$ is a smooth complete intersection of dimension at least~$3$,
then $\Br(X)/\Br_0(X)=0$ (see~\cite[Appendix~A]{poonenvoloch}), so that $X(\A_k)^{\Br(X)}=X(\A_k)$.  Thus, if in addition~$X$ is Fano,
Conjecture~\ref{conj:ctrc} predicts that $X(k)$ should be dense in
$X(\A_k)$. In other words, the condition $X(k_v)\neq\emptyset$ for all $v\in\Omega$
should imply the existence of a rational point, and when a rational point exists,
the image of~$X(k)$ in $\prod_{v\in S} X(k_v)$ should be dense
for every finite subset $S \subset \Omega$.

(vi) Even when~$X$ is rationally connected, ``computing'' $\Br(X)/\Br_0(X)$ is seldom an easy task, be it understood
in the weak sense of determining the isomorphism class of this finite abelian group
or in the strong sense of explicitly representing, by central simple algebras
over~$k(X)$, a finite set of classes of $\Br(X)$ which generate the quotient.
As the work of
 Uematsu~\cite{uematsudiagonalcubic}
clearly demonstrates,
the stronger problem, whose solution is essential
for evaluating the Brauer--Manin set, is not of a purely algebraic and geometric nature.
The required arithmetic input can be very delicate
(see~\cite{weinormtori} for a striking example).
Nevertheless, in a number of interesting cases, there exist practical algorithms that can solve this question
for
a given~$X$
and output a concrete description of the Brauer--Manin set,
thus deciding, in particular, whether it is empty,
or equal to $X(\A_k)$, or neither.
Such algorithms have been implemented for many del Pezzo surfaces
(all del Pezzo surfaces of degree~$4$, see~\cite{bbfl}, \cite{vavvertical};
most cubic surfaces, see~\cite{ctks}, \cite{elsenhansjahnelcubic1}, \cite{elsenhansjahnelcubic2};
some del Pezzo surfaces of degree~$2$, see~\cite{corndp2}).
Kresch and Tschinkel~\cite{kreschtschinkelalg} have given an algorithm
which can be applied to a geometrically
rational variety as soon as a sufficiently explicit description of
its geometric Picard group is known
(see \emph{op.\ cit.}\ for a precise statement).

(vii)
Let $\Br_1(X)=\Ker(\Br(X)\to\Br(X\otimes_k\bark))$
denote the \emph{algebraic} Brauer group.
If~$X$ is geometrically rational,
then
$\Br_1(X)=\Br(X)$,
as follows from Remark~\ref{rmks:conjct}~(iv)
and from the vanishing of $\Br(\P^n_{\bark})$.
\emph{Transcendental} elements of~$\Br(X)$ (\emph{i.e.}, elements which do not belong to $\Br_1(X)$)
are more difficult to exhibit and to exploit than algebraic ones.
However, they cannot be ignored
in the formulation of Conjecture~\ref{conj:ctrc}:
there exist smooth and proper rationally connected threefolds~$X$ over~$\Q$
such that
\begin{align}
\emptyset=X(\Q)=X(\A_\Q)^{\Br(X)}
\subset X(\A_\Q)^{\Br_1(X)}=X(\A_\Q)\neq\emptyset
\end{align}
(see~\cite{hararitranscendant}).
Another example of a rationally connected variety~$X$ over a number field~$k$
such that $X(\A_k)^{\Br(X)}\neq X(\A_k)^{\Br_1(X)}$
is given in~\cite{demarchelucchinineftin}.
Transcendental elements and their influence
on the Brauer--Manin set
have received a lot of attention
for other classes of varieties as well (see~\cite{witttransc},
\cite{ieronymou},
\cite{hassettvarillyvarilly},
\cite{preu},
\cite{hassettvarilly},
\cite{ieronymouskorobogatov},
\cite{newtontransc},
\cite{creutzviray},
\cite{mckstva} for examples of~$K3$ and Enriques surfaces
for which transcendental elements play a role
in the Brauer--Manin set).
\end{rmks}

For more information on the Brauer group, we refer the reader
to~\cite{ctnotesbrauer}.

\subsection{Rational points and \'etale covers}
\label{subsec:etalecovers}

When~$X$ is not simply connected, considering the Brauer--Manin
sets of all étale covers of~$X$ leads to more precise information than
considering the Brauer--Manin
set of~$X$ alone.
This idea is originally due to Skorobogatov~\cite{skorobeyond},
who managed, in this way,
to give the first example of a smooth and proper variety~$X$ over~$k$
such that $X(k)=\emptyset$ but $X(\A_k)^{\Br(X)}\neq\emptyset$
(a bielliptic surface).
Let us explain the underlying mechanism.

\begin{prop}
\label{prop:twist}
Let~$\barY$ be an irreducible finite étale cover
of~$X\otimes_k\bark$,
Galois over~$X$.
Let $x \in X(k)$.
There exists a morphism $\pi:Y\to X$
such that the schemes $Y\otimes_k\bark$ and~$\barY$
are isomorphic
over $X \otimes_k \bark$
and
 such that $x \in \pi(Y(k))$.
\end{prop}

By ``Galois over~$X$'' we mean that the function field extension $\bark(\barY)/k(X)$ is Galois.

\begin{proof}
Let $\pibar:\barY\to X\otimes_k \bark$ denote the structure morphism. Fix $y \in \pibar^{-1}(x)$.
The group $G=\Aut(\barY/X)$ acts on the finite set $\pibar^{-1}(x)$ since $x \in X(k)$.
Let~$H$ denote the stabiliser of~$y$ and $\pi:Y\to X$ the normalisation of~$X$ in the field~$\bark(\barY)^H$.
As the action of $N=\Aut(\barY/X\otimes_k \bark)$ on $\pibar^{-1}(x)$ is simply transitive, we have $N \cap H=1$ and $G=NH$.
It follows that $Y\otimes_k\bark$ and~$\barY$ are canonically isomorphic over $X \otimes_k \bark$.
As~$y$ is fixed by $H=G/N=\Gal(\bark/k)$, its image in~$Y$ is a rational point (above~$x$).
\end{proof}

\newcommand{\citepoonen}{Poonen~\cite{pooneninsufficiency}\footnote{The definition given
in~\cite{pooneninsufficiency} allows non-connected covers but involves twists of torsors under finite non-abelian groups (as in~\cite{harskononab}). It can be shown to be equivalent to the definition given here by slightly adapting
the arguments used in the proofs of~\cite[Prop.~3 and Prop.~7]{hararistix}.}}
\begin{defn}[\citepoonen]
\label{defn:etbr}
The \emph{étale Brauer--Manin set} of~$X$ is the set
\begin{align*}
X(\A_k)^{\et,\Br} = \bigcap_{\barY}\bigcup_{\vphantom{\barY\bark}\pi:Y\to X} \pi\big(Y(\A_k)^{\Br(Y)}\big)\rlap{\text{,}}
\end{align*}
where~$\barY$ runs over the irreducible finite étale covers of $X \otimes_k\bark$
that are Galois over~$X$,
where $\pi:Y\to X$ runs over the morphisms
such that $Y\otimes_k\bark$ and $\barY$ are isomorphic
over~$X \otimes_k \bark$,
and where the unions and the intersections are taken inside~$X(\A_k)$.
\end{defn}

The sequence~\eqref{eq:seqinclusions} can now be refined to the following sequence of
inclusions:
\begin{align}
\label{eq:seqinclusionset}
X(k)\subseteq X(\A_k)^{\et,\Br} \subseteq  X(\A_k)^{\Br(X)} \subseteq X(\A_k)\rlap{\text{.}}
\end{align}
The first inclusion holds by
Proposition~\ref{prop:twist};
the second one is clear
(consider the identity map $\pi:X\to X$).
We note, in addition, that $X(\A_k)^{\et,\Br}$ is a closed subset of $X(\A_k)^{\Br(X)}$.
Indeed, in the definition of the étale Brauer--Manin set,
it can be shown that for a given~$\barY$,
there are only finitely many $\pi:Y\to X$ such that $Y(\A_k)\neq\emptyset$
(see \cite[Prop.~4.4]{harskononab};
this relies on Hermite's finiteness theorem and, in a critical manner, on the properness of~$X$).

The bielliptic surface from~\cite{skorobeyond}
satisfies
$X(\A_k)^{\et,\Br}=\emptyset$
while $X(\A_k)^{\Br(X)}\neq\emptyset$.
Many other examples of non-simply connected, smooth, proper varieties are known for which
the étale Brauer--Manin set is strictly smaller than
the Brauer--Manin set
(in particular, rational points are dense neither in the Brauer--Manin set
nor in the modified Brauer--Manin set);
see \cite[\textsection\textsection5--6]{harariweakapprox} (combined with \cite{demarcheetale}),
\cite{basilesko},
\cite[\textsection3.3]{harskoenriques},
\cite{vavenriques},
\cite{bbmpvenriques}.
Of course, if the variety~$X$ is simply connected,
then $X(\A_k)^{\et,\Br}=X(\A_k)^{\Br(X)}$.

\subsection{Further obstructions}
\label{subsec:furtherobstructions}

Poonen~\cite{pooneninsufficiency} gave the first example of a smooth and proper
variety~$X$ over~$k$ such that $X(k)=\emptyset$ but
$X(\A_k)^{\et,\Br}\neq\emptyset$.
The idea is quite simple and may be summarised as follows.
Let~$C$ be a smooth, proper, geometrically irreducible curve over~$\Q$ with a unique rational point $c \in C(\Q)$.
Poonen constructs a smooth, projective threefold~$X$ and a morphism $f:X \to C$
whose fibre above~$c$ is the (smooth) surface
of Example~\ref{ex:isko}
(a surface such that $X_c(\A_\Q)\neq\emptyset$ but $X_c(\Q)=\emptyset$).
It follows, already, that $X(\Q)=\emptyset$.
The morphism~$f$ is defined as the composition
of a conic bundle $X \to \P^1_\Q \times C$, smooth over the complement of a smooth curve dominating
both~$\P^1_\Q$ and~$C$,
with the
second projection $\P^1_\Q \times C\to C$.
By a standard computation explained
in~\cite[\textsection5]{pooneninsufficiency}
or in~\cite[\textsection2]{ctsolides},
this description implies
that
for any finite étale $D \to C$, the map $f^*:\Br(D) \to \Br(X \times_C D)$
is surjective.
To prove that
$X(\A_\Q)^{\et,\Br}\neq\emptyset$,
we now check the inclusion $X_c(\A_\Q)\subseteq X(\A_\Q)^{\et,\Br}$.
An irreducible finite étale cover $\barY \to X\otimes_k\bark$ that is Galois over~$X$
 comes, by base change,
from some $\barD \to C \otimes_k\bark$ that is
Galois over~$C$,
as the fibres of $X \to \P^1_\Q \times C$ are simply connected.
By Proposition~\ref{prop:twist},
it even comes from a cover
$D \to C$ such that~$c$
lifts to $d \in D(\Q)$.
Letting $Y=X \times_C D$,
the surjectivity of
$f^*:\Br(D) \to \Br(Y)$
and the projection formula
show that $X_c(\A_\Q)=Y_{d}(\A_\Q) \subseteq Y(\A_\Q)^{\Br(Y)}$, hence the claim.

Thus, the étale Brauer--Manin set is not fine enough to explain why an arbitrary smooth and proper variety
over~$k$ can be devoid of rational points.
This was to be expected.
Indeed,
by Remark~\ref{rmks:conjct}~(v) and the weak Lefschetz theorem,
if $X\subset \P^n_k$ is a smooth complete intersection of dimension at least~$3$,
then $X(\A_k)^{\et,\Br}=X(\A_k)$;
on the other hand, conjectures
of Bombieri and Lang predict that as soon as~$X$ is of general type,
rational points should not be dense for the Zariski topology,
\emph{a fortiori} not dense in~$X(\A_k)$.
Based on the work of Masuda and Noguchi~\cite{masudanoguchi},
Sarnak and Wang~\cite{sarnakwang} have shown that these conjectures even imply the existence
of smooth hypersurfaces $X \subset \P^4_\Q$ (of degree~$1130$)
such that $X(\A_\Q)\neq\emptyset$ and $X(\Q)=\emptyset$.

It would be extremely interesting to formulate a computable necessary
condition for the existence of a rational point
on smooth hypersurfaces $X\subset \P^n_k$ of dimension at least~$3$,
not implied by the existence of an adelic point.
Grothendieck's birational section conjecture~\cite{gtof}
provides a condition which is necessary, and conjecturally sufficient, but hardly computable:
for a rational point to exist,
the projection of absolute Galois groups $G_{k(X)}\to G_k$ must admit a
continuous homomorphic section
(see~\cite[Lemma~3.1]{ewbirab}).
The existence of such a section does imply, however,
the existence of an adelic point
(see~\cite[\textsection2.2]{koenigsmannsection};
more generally, for any smooth, proper and irreducible variety~$X$ over~$k$,
the existence of such a section can be shown to imply that $X(\A_k)^{\et,\Br}\neq\emptyset$,
see \cite[\textsection\textsection11.6--11.7]{stixbook} and the proof of~\cite[Th.~15]{hararistix}).
It remains a challenge to exhibit a smooth hypersurface $X \subset \P^n_k$ of dimension at least~$3$
such that $X(\A_k)\neq\emptyset$ and such that
the projection $G_{k(X)}\to G_k$ does not admit a continuous homomorphic section.

For arbitrary smooth and proper varieties~$X$ over~$k$, two further theories, leading to two
subsets $X(\A_k)^{\mathrm{desc}}$
and $X(\A_k)^h$ of $X(\A_k)^{\Br(X)}$ containing~$X(k)$, have been investigated:
non-abelian descent, by Harari and Skorobogatov (see~\cite{harariweakapprox},
\cite{harskononab},
\cite{hararinonabelian}),
and the étale homotopy obstruction, by Harpaz and Schlank (see~\cite{harpazschlank}).
In the end, however, it turned out
that $X(\A_k)^{\mathrm{desc}}=X(\A_k)^h=X(\A_k)^{\et,\Br}$ always
(see~\cite{demarcheetale},
\cite{skoroequivalent},
and
\cite[Th.~9.136]{harpazschlank}).

Poonen's ideas
from~\cite{pooneninsufficiency}
have recently been amplified
to yield more examples of smooth and proper varieties~$X$ over~$k$
such that $X(k)=\emptyset$ but $X(\A_k)^{\et,\Br}\neq\emptyset$.
Examples are now known among surfaces
(over any number field; see~\cite{harpazskorobogatovetbr}),
among conic bundle surfaces over an elliptic curve (over a real quadratic field;
see~\cite{cps}),
among varieties with trivial Albanese variety
(see~\cite{smeetsinsufficiency}), and, if the $abc$ conjecture is true,
even among simply connected varieties (\emph{op.\ cit.}, \textsection4).

\subsection{\texorpdfstring{$R$\nobreakdash-equivalence}{R-equivalence} of rational points}
\label{subsec:requivalence}

Following Manin~\cite{manincubicforms},
we say that
two rational points $x,y \in X(k)$
of a proper variety~$X$ over~$k$
are \emph{directly $R$\nobreakdash-equivalent} if there exists a map $\phi:\P^1_k\to X$ such that $\phi(0)=x$, $\phi(\infty)=y$ and we define \emph{$R$\nobreakdash-equivalence} as the equivalence relation on~$X(k)$ generated by direct $R$\nobreakdash-equivalence.
The set of $R$\nobreakdash-equivalence classes is denoted $X(k)/R$.

An important open question is whether
 $X(k)/R$ is finite
when~$X$ is a smooth and proper rationally connected variety over a number field~$k$.
This question is discussed in~\cite[\textsection10]{ctdegenerescences},
to which we refer for a list of known results.
According to Kollár~\cite{kolllocalfield} and to Kollár and Szabó~\cite{kollarszabo},
for such~$X$,
the set
$X(k_v)/R$ is finite for every place~$v$ of~$k$
and has cardinality~$1$ for all but finitely many~$v$, so that $\prod_{v \in\Omega} X(k_v)/R$
is finite. However, the natural map $X(k)/R\to \prod_{v \in\Omega}X(k_v)/R$
need not be injective (see~\cite[\textsection{}V.2]{sansucapropos})
and it is not known how to control its fibres in general.
We note that
if $X(k)\neq\emptyset$,
then
 Conjecture~\ref{conj:ctrc}
together with
the finiteness of $X(k)/R$
would imply,
by Remark~\ref{rmks:conjct}~(iii), 
that~$X$ contains a very large number of rational curves defined over~$k$.
It is not known whether any smooth and proper
rationally connected variety~$X$ over~$k$
with $X(k)\neq\emptyset$
contains even one such curve.

\subsection{Beyond rationally connected varieties, but not too far beyond}
\label{subsec:nonrc}

As we have explained in~\textsection\ref{subsec:bmo}
and in~\textsection\ref{subsec:furtherobstructions},
although existence and approximation properties of rational points are conjectured
to be determined by
the Brauer--Manin set
on rationally connected varieties,
it is clear that
for simply connected varieties of general type,
 other mechanisms must come
into play.
What happens in between?

\subsubsection{Curves}

Suppose~$X$ is a (smooth, proper, geometrically irreducible) curve of genus~$g$
over~$k$.  When $g=0$, the anti-canonical embedding presents~$X$ as a conic;
we have $\Br(X)=\Br_0(X)$,
and~$X(k)$ is dense in $X(\A_k)^{\Br(X)}=X(\A_k)$
by the Hasse--Minkowski theorem.
Using the work of Cassels on elliptic curves,
Manin~\cite[\textsection6]{maninicm} showed that $X(\A_k)^{\Br(X)}\neq\emptyset$ implies $X(k)\neq\emptyset$
if
 $g=1$
and
 the Tate--Shafarevich group of the Jacobian of~$X$
is finite (a widely believed conjecture).
A number of examples in higher genus led Scharaschkin and Skorobogatov to suggest that this
implication might in fact hold for all smooth and proper curves
(see~\cite[\textsection6.2]{skobook}).
When~$g\geq 1$, the set~$X(k)$ can be simultaneously finite and non-empty, in which case
it cannot be dense in~$X(\A_k)^{\Br(X)}$
(see Remark~\ref{rmks:conjct}~(iii)).
Nevertheless, conjectures of Poonen~\cite{poonenheuristics}
and Stoll~\cite{stoll},
backed by
a probabilistic heuristic
and extensive numerical evidence (see~\cite{stollrationalpointsoncurves}),
predict that~$X(k)$ should be dense in~$X(\A_k)_\poonen^{\Br(X)}$ for any smooth and proper curve~$X$ over~$k$.
For curves over a global field of positive characteristic,
Poonen and Voloch~\cite{poonenvolochbmff}
have shown, under a very weak technical hypothesis on the
Jacobian,
that $X(k)$ is dense in $X(\A_k)^{\Br(X)}$
(see also~\cite{hararibourbakipoonenvoloch}).

\subsubsection{Abelian varieties}

If~$X$ is a principal homogeneous space of an abelian variety over~$k$
whose Tate--Shafarevich group is assumed to be finite,
the set~$X(k)$ is dense in~$X(\A_k)_\poonen^{\Br(X)}$ (see~\cite{wang}, \cite{harariconnected}).  A rather surprising by-product of the proof is the equality $X(\A_k)^{\Br(X)}=X(\A_k)^{\Br_1(X)}$ for such~$X$
 (see Remark~\ref{rmks:conjct}~(vii)). If~$X$ is an abelian variety of dimension at least~$2$,
very often $\Br(X)\neq \Br_1(X)$.  Curves, on the other hand,
satisfy $\Br(X)=\Br_1(X)$, in view of Tsen's theorem.

\subsubsection{$K3$ surfaces}
\label{subsubsec:k3surfaces}

A number of conditional positive results for certain elliptic~$K3$ surfaces
(see~\textsection\ref{subsubsec:pencilsofabvar} below)
have led to the idea that the set~$X(k)$ might always be dense
in $X(\A_k)^{\Br(X)}$,
or perhaps in $X(\A_k)^{\Br(X)}_\bullet$, when~$X$ is a~$K3$ surface.
Evidence is scarce:
when non-empty, the set~$X(k)$ is not known
to be dense in $X(\A_k)^{\Br(X)}_\bullet$
for a single~$K3$ surface~$X$ over a number field~$k$.
Carrying out numerical experiments is rendered difficult by two phenomena.
First, rational points of~$K3$ surfaces tend to have large height rather quickly
(see~\cite{vanluijkslides} for a discussion of the conjectural growth of the number of rational
points
of bounded height on~$K3$ surfaces) and secondly, for~$K3$ surfaces, and more generally
for varieties with $H^2(X,\sO_X)\neq 0$,
it is a particularly challenging task, in practice,
to make the transcendental elements
of~$\Br(X)$ explicit,
or even to simply bound the quotient $\Br(X)/\Br_1(X)$
when it is known to be finite (see Remark~\ref{rmks:conjct}~(ii)).
See however~\cite{elsjahneltransc}
for a partial experiment with Kummer surfaces over~$\Q$.

Skorobogatov and Zarhin~\cite{skozarhinfiniteness}
have proved the finiteness of $\Br(X)/\Br_0(X)$
for all~$K3$ surfaces~$X$ over~$k$;
this finiteness was not previously known in any example.
So far, this quotient has been computed, or bounded, only for certain Kummer surfaces,
see~\cite{skorobogatovzarhinkummer},
\cite{ieronymouskorobogatovzarhin},
\cite{ieronymouskorobogatov}.
In addition, Hassett, Kresch and Tschinkel~\cite{hassettkreschtschinkel} have given
an algorithm which takes, as input,
a~$K3$ surface~$X$ of degree~$2$ (\emph{i.e.}, a double
cover of~$\P^2_k$ ramified along a sextic curve), and outputs a bound for the order
of~$\Br(X)/\Br_0(X)$. In its present form, however, it is not practical.

Kresch and Tschinkel~\cite{kreschtschinkelsurfaces}
have shown that
whenever~$X$ is a~$K3$ surface for which a bound on the order of $\Br(X)/\Br_0(X)$ is known,
the set
 $X(\A_k)^{\Br(X)}$
can be computed algorithmically.
By~\cite{ieronymouskorobogatovzarhin},
this applies, in particular,
to diagonal quartic surfaces over~$\Q$, \emph{i.e.}, surfaces in~$\P^3_\Q$ defined by an equation of the form $\sum a_ix_i^4=0$
for $a_0,\dots,a_3\in \Q^*$.
It nevertheless remains a challenge to implement, on an actual computer, an algorithm which
takes, as input, such an equation, and outputs yes or no according
to whether the set $X(\A_\Q)^{\Br(X)}$ is empty.
Here,
the difficulty does not lie in bounding the group $\Br(X)/\Br_0(X)$,
whose order always divides $2^{25} \cdot 15$
(see~\cite{ieronymouskorobogatovzarhin}
and~\cite{ieronymouskorobogatov}), but in representing
its elements by central simple algebras over~$\Q(X)$.
The analogous problem for $X(\A_\Q)^{\Br_1(X)}$ was solved by Bright~\cite{brightdiagonalalgebraic}.  It is known that transcendental elements cannot be ignored in this context: there exist diagonal
quartic surfaces~$X$ over~$\Q$ such that $X(\A_\Q)^{\Br(X)}\neq X(\A_\Q)^{\Br_1(X)}$
(see~\cite{ieronymou},
\cite{preu},
\cite{ieronymouskorobogatov}).
There also exist~$K3$
surfaces~$X$ over~$\Q$
such that
\begin{align}
\emptyset=X(\Q)=X(\A_\Q)^{\Br(X)}
\subset X(\A_\Q)^{\Br_1(X)}=X(\A_\Q)\neq\emptyset\rlap{\text{,}}
\end{align}
even among~$K3$ surfaces of degree~$2$ with geometric Picard number~$1$
(see~\cite{hassettvarilly},
which rests on~\cite{vanluijk} and
\cite{ejpicardreduction}
to prove that the geometric Picard number is~$1$).

A conjecture of Campana
\cite[\textsection4.3]{campanaoverview}
predicts
Zariski density of rational points,
on~$K3$
surfaces over number fields,
after a finite extension of the ground field.
For some~$K3$ surfaces of geometric Picard number at least~$2$,
this is known (see~\cite{hassettpotentialdensity}, which also
discusses,
in~\textsection11, surfaces of geometric Picard number~$1$).
In addition, density of~$X(k)$ in~$X(k_v)$ for a fixed $v \in \Omega$ is known
for some~$K3$ surfaces admitting two elliptic fibrations; see~\cite{loganmckinnonvanluijk},
\cite{swinnertondyerdensity}, \cite{pannekoek}.

\subsection{Zero-cycles}
\label{subsec:zerocycles}

The definitions and questions of~\textsection\ref{subsec:bmo} still make sense if rational points are replaced with zero-cycles of a given degree.
If~$P$ is a closed point of~$X$, we denote by~$k(P)$ its residue field; it is a finite extension of~$k$, equal to~$k$ if and only if~$P$ is rational.
Recall that a \emph{zero-cycle} on~$X$
is a formal $\Z$\nobreakdash-linear combination of closed points and that the degree of $z=\sum n_i P_i$
is by definition $\deg(z)=\sum n_i [k(P_i):k]$.
Thus, any rational point is a zero-cycle of degree~$1$, and a zero-cycle of degree~$1$ exists if and only if the gcd of the degrees of the closed points of~$X$ is equal to~$1$,
 a condition weaker than the existence of a rational point.
The map $\Br(X) \times X(\A_k) \to \Q/\Z$, $(\alpha,P_v)\mapsto\sum_{v\in\Omega}\inv_v\mkern2.5mu\alpha(P_v)$
considered in Definition~\ref{def:bmset} extends to a canonical pairing
\begin{align}
\label{eq:globalpairing}
\Br(X) \times \prod_{v\in\Omega}\CH_0(X\otimes_k k_v) \to \Q/\Z\rlap{\text{,}}
\end{align}
where~$\CH_0$ denotes the Chow group of zero-cycles up to rational equivalence.
By the global reciprocity law, the image of the diagonal map
$\CH_0(X) \to \prod_{v\in\Omega}\CH_0(X\otimes_k k_v)$ is contained in the right
kernel of~\eqref{eq:globalpairing}; thus, a necessary condition for the existence of a zero-cycle of degree~$1$ on~$X$ is the existence of a collection of zero-cycles $(z_v)_{v\in\Omega}$ in the right kernel of~\eqref{eq:globalpairing} such that $\deg(z_v)=1$ for all $v\in\Omega$.
When such a collection exists, we say that there is no \emph{Brauer--Manin obstruction to the existence of a zero-cycle of degree~$1$ on~$X$}.
This condition is implied by, but is in general weaker than, the condition $X(\A_k)^{\Br(X)}\neq\emptyset$.

\newcommand{\citeconjksct}{Kato, Saito~\cite[\textsection7]{katosaitocontemp}, Colliot-Thélène~\cite[\textsection1]{ctbordeaux}}
\begin{conj}[\citeconjksct]
\label{conj:e1}
Let~$X$ be a smooth, proper, geometrically irreducible variety defined over a number field~$k$.
There exists a zero-cycle of degree~$1$ on~$X$
if there
is no Brauer--Manin obstruction to the existence
of a zero-cycle of degree~$1$ on~$X$.
\end{conj}

The above conjecture, first formulated by Colliot-Thélène and Sansuc~\cite{ctsansuc}
in the
case of geometrically rational surfaces, makes no assumption on the geometry of~$X$.
This is in marked contrast with Conjecture~\ref{conj:ctrc}
and the expected behaviour of rational points on varieties of general type
(see~\textsection\ref{subsec:furtherobstructions}).

There is a natural strengthening of Conjecture~\ref{conj:e1},
referred to as Conjecture~(E),
which includes an approximation condition for
zero-cycles analogous to the density of~$X(k)$ in $X(\A_k)^{\Br(X)}_\poonen$ and which should
also hold regardless of the geometry of~$X$.
Its precise formulation can be found in~\cite[\textsection1.1]{wittdmj}, or, in an essentially equivalent form,
in~\cite[\textsection1]{ctbordeaux}.

For curves,
Conjectures~\ref{conj:e1} and~(E) are known to follow from the finiteness of the Tate--Shafarevich group
of the Jacobian
(see~\cite{saito},
\cite[\textsection3]{ctconj},
\cite{erikssonscharaschkin},
\cite[Remarques~1.1~(iv)]{wittdmj}).
For abelian varieties of dimension greater than~$1$, Conjecture~(E) is open.
For rationally connected varieties,
Liang~\cite{liangarithmetic}
has shown that if~$X(k')$ is dense in $X(\A_{k'})^{\Br(X\otimes_k k')}$ for all finite extensions~$k'/k$,
then~$X$ satisfies Conjectures~\ref{conj:e1} and~(E).
The proof is based on the fibration method and will be sketched in~\textsection\ref{subsubsec:zerocyclesandfibrations} below.

\subsection{What about integral points?}
\label{subsec:integralpoints}

It was recently recognised by Colliot-Thélène and Xu~\cite{ctxu},
in a foundational work on homogeneous spaces of semi-simple
algebraic groups,
that the concepts discussed in~\textsection\ref{subsec:bmo} are relevant not only for
the study of rational points,
but also for that of integral points.

Let us drop the hypothesis, introduced
at the beginning
of~\textsection\ref{sec:definitionsandquestions},
that~$X$ is proper.
Until the end of~\textsection\ref{subsec:integralpoints}, we only assume that~$X$ is a smooth,
geometrically irreducible, separated variety over~$k$.
Let~$\sX$ be a separated scheme of finite type over the ring of integers~$\sO_k$ of~$k$, with generic fibre~$X$.
For any finite place~$v$ of~$k$,
a $k_v$\nobreakdash-point of~$X$ is \emph{integral}
with respect to~$\sX$
if it belongs to $\sX(\sO_v)\subseteq X(k_v)$.
For any finite subset $S \subset \Omega$ containing~$\Omega_\infty$,
a rational point of~$X$ is $S$\nobreakdash-\emph{integral}
(or \emph{integral}, when $S=\Omega_\infty$)
if it is integral at every finite place not in~$S$,
or, equivalently, if it belongs
to the subset $\sX(\sO_{k,S})\subseteq X(k)$,
where $\sO_{k,S}$ denotes the ring of $S$\nobreakdash-integers of~$k$.
We set $\sX(\A_{k,S})=\prod_{v \in S} X(k_v) \times \prod_{v \in \Omega \setminus S} \sX(\sO_v)$.
The set of \emph{adelic points} of~$X$ is $X(\A_k) = \varinjlim \sX(\A_{k,S})$,
where the limit ranges over the finite subsets $S \subset\Omega$ containing~$\Omega_\infty$, ordered by inclusion.
This set does not depend on the choice of~$\sX$ since any two models of~$X$ become
isomorphic after inverting finitely many primes.
Following Weil~\cite{weiladeles}, we endow $\sX(\A_{k,S})$ with its natural product topology
and $X(\A_k)$ with the direct limit topology.
Usually,
this topology is strictly finer than the one induced by the inclusion
$X(\A_k) \subseteq \prod_{v \in \Omega}X(k_v)$;
the subsets
$\prod_{v \in S} U_v \times \prod_{v \in \Omega \setminus S} \sX(\sO_v)$,
for $S \subset \Omega$ finite containing~$\Omega_\infty$ and $U_v \subseteq X(k_v)$ arbitrary open subsets,
form a basis of open sets of~$X(\A_k)$.
We refer the reader to~\cite{conradadeles} and~\cite{lorscheidsalgado} for
a detailed discussion of this topology and of its relationship with Grothendieck's definition
of~$X(\A_k)$ as the set of $k$\nobreakdash-morphisms $\Spec(\A_k)\to X$.

In general, the study of integral points
is more difficult than that of rational points on the same variety,
although if~$X$ is proper, we can take~$\sX$ to be proper too, in which case rational
and integral points are the same.

As in the proper case, we have
a diagonal embedding $X(k) \subseteq X(\A_k)$
and a closed subset $X(\A_k)^{\Br(X)} \subseteq X(\A_k)$, defined exactly
as in~\textsection\ref{subsec:bmo}.
The global reciprocity law still implies
that $X(k) \subseteq X(\A_k)^{\Br(X)}$.
This inclusion now has implications for integral points on~$X$;
for instance,
it follows that if
 $S \subset \Omega$
is a finite subset containing~$\Omega_\infty$
such that $X(\A_k)^{\Br(X)} \cap \sX(\A_{k,S})=\emptyset$,
then $\sX(\sO_{k,S})=\emptyset$.

\begin{example}
There do not exist integers $x,y \in \Z$ such that $x^2-25y^2=-1$,
even though this equation admits a rational solution
as well as a solution in $\Z/n\Z$ for every $n \geq 1$
since $(x,y)=(0,\frac 1 5)$ and $(x,y)=(\frac 3 4, \frac 1 4)$ are solutions
in $\Z[\frac 1 5]$ and $\Z[\frac 1 2]$, respectively.
Indeed, let $\sX \subset \A^2_\Z$ be the closed subscheme defined by this equation
and $X=\sX\otimes_\Z\Q$.  One can check that the quaternion
algebra $(x-5y,5) \in \Br(\Q(X))$ extends (uniquely) to a class $\alpha\in\Br(X)$,
that $\alpha(P_v)=0$ for any $P_v \in \sX(\Z_v)$ and any prime number $v \neq 5$,
that $\alpha(P_\infty)=0$ for any $P_\infty \in X(\R)$ and that
$\inv_5\mkern2.5mu\alpha(P_5)=\frac{1}{2}$
for any $P_5 \in \sX(\Z_5)$.
Thus $X(\A_\Q)^{\Br(X)} \cap \sX(\A_{\Q,\Omega_\infty})=\emptyset$
and hence $\sX(\Z)=\emptyset$: there is a \emph{Brauer--Manin obstruction to the existence of
an integral point on~$\sX$}.

Of course, 
using the factorisation $x^2-25y^2=(x-5y)(x+5y)$
together with
the fact that~$-1$ and~$1$ are the only units of~$\Z$,
it is straightforward to check directly that integral solutions do not exist;
but the above argument, based on~$\alpha$, remains valid over
any number field~$k$ of odd degree in which~$5$ is inert,
so that $\sX(\sO_k)=\emptyset$ for any such~$k$, however unwieldy the group of units of~$\sO_k$ may be.
\end{example}

When the variety~$X$ is not proper, the set $X(k)$ is rarely dense in $X(\A_k)^{\Br(X)}$:
this density fails already for the affine line, since in this case $X(\A_k)^{\Br(X)}=X(\A_k)$
while~$k$ is never dense in~$\A_k$.
A more interesting property is obtained by ignoring approximation conditions at
a finite set $S \subset \Omega$ of places.

\begin{defn}
\label{def:densityoff}
If~$\pi$
denotes the projection $\pi:X(\A_k) \to \prod_{v\in\Omega\setminus S}X(k_v)$,
the \emph{adelic topology off~$S$} on $X(\A_k)$
is the topology
whose open subsets are the sets $\pi^{-1}(\pi(U))$ when~$U$ ranges over the (usual) open subsets
of $X(\A_k)$.  We refer to the density of an inclusion of subsets of~$X(\A_k)$
with respect to the induced topology as \emph{density off~$S$}.
\end{defn}

The density of~$X(k)$ in~$X(\A_k)$ off~$S$
is commonly referred to as the property of \emph{strong approximation off~$S$},
as opposed to \emph{weak approximation}, which refers to the density of~$X(k)$ in $\prod_{v \in \Omega}X(k_v)$.
The Chinese remainder theorem essentially states that the affine line over~$\Q$ satisfies
strong approximation off the real place.  In fact, affine space of any dimension over any number field
satisfies strong approximation off any non-empty finite set of places.
For more general varieties, the Brauer--Manin obstruction can again play a role:
if $X(\A_k)^{\Br(X)}$ is not dense in~$X(\A_k)$ off~$S$, then strong approximation off~$S$ must fail.

\newcommand{\citectw}{\cite[Exemple~5.6]{ctw}}
\begin{example}[\citectw]
There do not exist integers $x,y,z\in\Z$ such that $x^3+y^3+2z^3=2$
and such that
$\congru{x+y}{2}{8}$ and $\congru{z}{2}{4}$,
even though such $x,y,z$ can be found in~$\Z_2$.
(In particular, strong approximation off the real place fails.)
Indeed, letting $\sX \subset \A^3_\Z$ denote the closed subscheme defined by
 $x^3+y^3+2z^3=2$
and $X=\sX\otimes_\Z\Q$, one can check that the quaternion algebra
$(2(x+y+2z), -3(x+y+2z)(x+y)) \in \Br(\Q(X))$ extends (uniquely) to a class
$\alpha\in\Br(X)$,
that $\alpha(P_v)=0$ for any $P_v \in \sX(\Z_v)$ and any prime number $v \neq 2$,
that $\alpha(P_\infty)=0$ for any $P_\infty \in X(\R)$, that
$\inv_2\mkern2.5mu\alpha(P_2)=\frac{1}{2}$
for any $P_2 \in \sX(\Z_2)$ which satisfies the two
congruence conditions
$\congru{x+y}{2}{8}$ and $\congru{z}{2}{4}$,
and that such~$P_2$ exist.
If we fix such a~$P_2$ and let $P_v=(0,0,1)$ for $v \in \Omega\setminus\{2\}$,
we obtain an element of $X(\A_\Q)$
which does not belong to
the closure
of $X(\A_\Q)^{\Br(X)}$
off~$S=\{\infty\}$.
This is a \emph{Brauer--Manin obstruction to strong approximation on~$X$ off the real place}.
\end{example}

The reader will find many more examples of Brauer--Manin obstructions to the existence of integral
points, or to strong approximation, in \cite[\textsection8]{ctxu}, \cite{kreschtschinkeltwoexamples},
\cite{ctw},
\cite{weixu},
\cite{weixumult},
\cite{gundlach}, \cite[\textsection7]{ctxu2},
\cite{weisumoftwosquares},
\cite[\textsection\textsection6--8]{jahnelschindler}.

Now that the various concepts relating integral points, strong approximation and the Brauer--Manin set
have been described and that examples have been given, we cannot but observe that
even just formulating a sound analogue of Conjecture~\ref{conj:ctrc}
for non-proper varieties remains a delicate task, not satisfactorily solved to this day.
In contrast with what happens for rational points, it is very easy to give
examples of finite type separated schemes over~$\Z$ which
have
neither an integral point nor a Brauer--Manin obstruction to
the existence of an integral point.  Such an example is the closed subscheme $\sX \subset \A^3_\Z$
defined by $2x^2+3y^2+4z^2=1$.
The non-existence of
integral points on~$\sX$
is a triviality;
on the other hand, it is proved in~\cite[Exemple~5.9]{ctw} that there is no Brauer--Manin obstruction to the existence of an integral point on~$\sX$.
What goes wrong in this example is that~$X(\R)$ is compact,
thus forcing the finiteness of the subset~$\sX(\Z)$.
Subtler examples of a similar type, based on archimedean phenomena,
can be found
in~\cite[\textsection4.2]{harpazcurious}
(no integral point; this example is also discussed in~\cite[\textsection2]{jahnelschindler})
and
in~\cite[\textsection6]{derenthalwei}
and~\cite[\textsection8]{jahnelschindler}
(defect of strong approximation off the real place).

In addition, the phenomena discussed 
in~\textsection\ref{subsec:etalecovers}
and~\textsection\ref{subsec:furtherobstructions} for rational points
do have their analogues for integral points.

The influence of étale covers on integral points is illustrated, over~$\Z$, by~\cite[Exemple~5.10]{ctw} and
\cite[Example~1.6]{weitorus}
(no integral point)
and by~\cite[p.~420]{hararivolochcurves} (failure of strong approximation off the real place, for a specific adelic point).
In these three examples,
the defect
is due to an étale Brauer--Manin obstruction and is not explained by the Brauer--Manin obstruction.

Poonen's use, in~\cite{pooneninsufficiency}, of a morphism to a curve with only finitely many rational points,
which we summarised in~\textsection\ref{subsec:furtherobstructions},
can also be adapted to the integral setting:
Cao and Xu~\cite[Example~5.2]{caoxu} exploit the projection
$X=(\A^1_\Q \times \Gm)\setminus \{(0,1)\} \to \Gm$
and the finiteness of $\Gm(\Z)=\{-1,1\}$
to exhibit,
 for certain adelic points on~$X$
lying in the fibre above~$1$,
 a failure of strong approximation off the real place which cannot be explained by a Brauer--Manin obstruction on~$X$ or by its étale variant.

The example we have just mentioned shows that
unlike the fundamental group or the Brauer group,
 qualitative properties of integral points
 can be sensitive to codimension~$2$ subsets;
indeed, it follows from Poitou--Tate duality that
the Brauer--Manin obstruction fully accounts for the defect of
strong approximation off the real place
 on the surface $\A^1_\Q \times \Gm$.
The following question, however, is open.
It demonstrates a second type of difficulties that are inherent to integral points.

\begin{question}
\label{q:codimdeux}
Let $S\subset \Omega$ be a finite subset.
Let~$X$ be a smooth and separated variety which satisfies strong approximation off~$S$.
If $Z \subset X$ denotes a closed subset of codimension at least~$2$,
does $X \setminus Z$
satisfy strong approximation off~$S$?
\end{question}

It was observed by Wei~\cite[Lemma~1.1]{weitorus} and by Cao and
Xu~\cite[Prop.~3.6]{caoxu} that Question~\ref{q:codimdeux} has an
affirmative answer when $X=\A^n_k$ or $X=\P^n_k$.

A very similar question for potential Zariski density of integral points was
raised and discussed by Hassett and
Tschinkel~\cite[Problem~2.13]{hassetttschinkelintegral}.  (``Potential''
means that we are considering $S$\nobreakdash-integral points where both an
enlargement of~$S$ and a finite extension of~$k$ are allowed.)

Despite the pitfalls related to archimedean places and to codimension~$2$ subsets,
significant progress was achieved over the last few years
in the study of strong approximation on various types of varieties.
This is still a burgeoning area of research;
instead of discussing it in depth,
we merely indicate some of its developments as we go along
(see~\textsection\ref{intpoints:galois}, \textsection\ref{intpoints:descent}, \textsection\ref{intpoints:fibration}).
In addition to the references given there, the reader may consult 
\cite{hararivolochcurves},
\cite{gundlach},
\cite{liuxu},
\cite{xufibrations}.

\section{Methods for rational and rationally connected varieties}
\label{sec:methods}

General methods that have been successful at solving particular
cases of the questions raised in~\textsection\ref{sec:definitionsandquestions}
for rationally connected varieties over number fields
fall under four headings:
analytic methods,
Galois cohomological methods,
the descent method,
and the fibration method.

Over global fields of positive characteristic,
one can also approach these questions
through geometry over finite fields,
as the recent work of Tian~\cite{tian} illustrates;
see also~\cite{hassetttschinkelquadricfibrations}.
For lack of space, we focus on number fields in these notes.

In the following sections
we review the four methods listed above, with an emphasis on
examples which showcase interplay between two or more of them.

\subsection{Analytic methods}

By ``analytic methods'', we mean the Hardy--Littlewood circle method, sieve methods,
and the recent developments in additive combinatorics due to Green, Tao and Ziegler.
(See~\cite{davenportbook}, \cite[\textsection16]{bakerbook}, \cite{zieglersurvey} for introductions to these topics.)
Instead of discussing their technical details,
we quote below a few results
which are meant to be representative of the range of applicability of these methods.
As the reader will observe,
this range is limited.
When analytic methods apply, however,
they yield results of a quantitative nature, more precise
than the existence or density statements discussed in this report;
in addition, even the qualitative implications of these results are often completely unapproachable by
the known algebro-geometric
methods. Such is the case, for instance, for Theorem~\ref{th:bhb} below with $\deg(X)\geq 4$ or 
for Example~\ref{ex:bms} below with $r\gg 0$.

\subsubsection{The circle method}

Rational points of a hypersurface in a projective space are amenable to analysis by the circle method
when the degree
of the hypersurface is small enough compared to its dimension.
Over~$\Q$, the following theorem was proved by
 Browning and Heath--Brown~\cite{bhb},
generalising earlier work of Birch~\cite{birch} which covered, in particular, the case of hypersurfaces.
The extension to number fields is due
to Skinner~\cite{skinner} and to
Frei and Madritsch~\cite{freimadritsch}.

\begin{thm}
\label{th:bhb}
Let $n\geq 1$.
Let~$k$ be a number field.
If $X \subset \P^n_k$ is a smooth and geometrically irreducible subvariety such that
\begin{align}
\label{eq:bounddimdeg}
\dim(X) \geq \left(\deg(X)-1\right)2^{\deg(X)}-1\rlap{\text{,}}
\end{align}
then~$X(k)$ is dense in $X(\A_k)$.
\end{thm}

According to a result of Bertram, Ein and Lazarsfeld~\cite{bertrameinlazarsfeld} used
in the proof
of Theorem~\ref{th:bhb}, any $X \subset \P^n_k$ subject to~\eqref{eq:bounddimdeg}
is a complete intersection,
so that Theorem~\ref{th:bhb} is really a statement about smooth
complete intersections in projective space.
In view of Remark~\ref{rmks:conjct}~(v),
this explains why
$X(\A_k)^{\Br(X)}=X(\A_k)$,
an equality which is necessary for Theorem~\ref{th:bhb} to hold.

As follows from
Remark~\ref{rmks:conjct}~(v),
if Conjecture~\ref{conj:ctrc} is true
and~$X$ is assumed from the start to be a complete intersection,
then the bound~\eqref{eq:bounddimdeg} can be lowered to
\begin{align}
\dim(X)\geq \max\mkern.5mu(3,\deg(X)-1)\rlap{\text{.}}
\end{align}
Thus~\eqref{eq:bounddimdeg} is, in principle, very far from being optimal.
For complete intersections of specific multidegrees, however,
improvements to the circle method can provide better results
than Theorem~\ref{th:bhb}.  For instance, if $X \subset \P^n_\Q$ is a smooth hypersurface
of degree~$d$ with $X(\A_\Q)\neq\emptyset$, it is known that $X(\Q)\neq\emptyset$
as soon as $d=3$ and $n\geq 8$,
instead of $n\geq 16$
(see~\cite{heathbrownten}, \cite{hooley}),
or $d=4$ and $n\geq 39$, instead of $n \geq 48$
(see~\cite{bhbquartic}, \cite{hanselmann}),
or $d=5$ and $n\geq 110$, instead of $n \geq 128$ (see~\cite{browningprendiville}).
We refer the reader to~\cite{browningsurvey} for a thorough discussion of
these and related results. We also note that the circle method
has been adapted to global fields of positive characteristic;
see~\cite{browningvishefq}.

\subsubsection{Sieve methods, additive combinatorics}
\label{subsubsec:sieve}

Before giving examples of what sieve methods and additive combinatorics can do,
let us introduce a convenient definition.  Let $L/k$ be a finite extension of number fields.
Denote the Weil restriction of scalars of~$\A^1_L$ from~$L$ down to~$k$
by $R_{L/k}(\A^1_L)$.
Let $s \geq 0$.  Let $f \in k[u_1,\dots,u_s]$.
The choice of a basis $\omega=(\omega_1,\dots,\omega_m)$ of~$L$ as a $k$\nobreakdash-vector space
determines an isomorphism $R_{L/k}(\A^1_L) \times \A^s_k \simeq \A^{m+s}_k$
as well as a polynomial equation $N_{L/k}(x_1\omega_1 +\dots+x_m\omega_m)=f(u_1,\dots,u_s)$
in $m+s$ variables.  The corresponding hypersurface of $R_{L/k}(\A^1_L) \times \A^s_k$ is independent of the choice of~$\omega$.
We shall refer to it as \emph{the affine variety defined by the equation $N_{L/k}(x)=f(u)$}.
In fact, we have already encountered
such a variety,
in Example~\ref{ex:isko},
with $L=\Q(i)$, $k=\Q$, $s=1$.
More generally, if we are given finite extensions~$L_i/k$ and polynomials $f_i \in k[u_1,\dots,u_s]$
for $i \in \{1,\dots,r\}$, we shall speak of \emph{the affine variety defined by the system $N_{L_i/k}(x_i)=f_i(u)$, $1\leq i\leq r$},
to refer to the corresponding closed subvariety of $\prod_{i=1}^r R_{L_i/k}(\A^1_{L_i}) \times \A^s_k$.

\begin{example}[\cite{irving}]
\label{ex:irving}
Let $p \geq 7$ be a prime number.
Let $k=\Q$, $L=\Q(2^{1/p})$ and $s=1$.
Let $f\in \Q[u]$ be a cubic polynomial which does not admit a root
in the cyclotomic field $\Q(\zeta_p)$.
Let~$X$ denote a smooth and proper model, over~$\Q$,
of the affine variety defined by the equation $N_{L/\Q}(x)=f(u)$.
Using sieve methods,
Irving~\cite{irving}
proves that if $X(\A_\Q)\neq\emptyset$, then $X(\Q)\neq\emptyset$
and in fact~$X(\Q)$ is dense in~$X$ for the Zariski topology.
\end{example}

Thanks to the projection $(x,u)\mapsto u$,
we may view
the affine variety considered in Example~\ref{ex:irving} as the total space of
a fibration, over~$\A^1_\Q$,
whose smooth fibres are torsors under the norm torus $N_{L/\Q}(x)=1$.
This variety does not have any other obvious geometric structure
which we might exploit to prove the assertion.

\begin{example}[\cite{bms}]
\label{ex:bms}
Let $f_1,\dots,f_r \in \Q[u_1,\dots,u_s]$
be pairwise non-proportional homogeneous linear polynomials,
with $s \geq 2$.
Let $L_1,\dots,L_r$ be quadratic extensions of~$\Q$.
Let~$X$ denote a smooth and proper model, over~$\Q$,
of the affine variety defined by the system $N_{L_i/\Q}(x_i)=f_i(u)$, $1\leq i\leq r$.
Using additive combinatorics,
Browning, Matthiesen and Skorobogatov~\cite[\textsection2]{bms} prove
that~$X(\Q)$ is dense in~$X(\A_\Q)$.
\end{example}

When $s\geq r$ and $f_1,\dots,f_r$ are linearly independent,
the affine variety considered in Example~\ref{ex:bms}
is isomorphic to an affine space.  On the other hand, if~$r$ is allowed to grow larger than~$s$,
this variety becomes geometrically very intricate; its only known structure is that of
the total space of a fibration, over~$\A^s_\Q$, whose smooth fibres are torsors under the product of the norm tori $N_{L_i/\Q}(x)=1$.
The point of Example~\ref{ex:bms} is that
additive combinatorics
remains applicable even when $s=2$ and~$r\gg 0$.

\subsection{Galois cohomological methods}
\label{subsec:galoiscoho}

If a smooth and proper variety~$X$ over~$k$ is birationally equivalent to a homogeneous space of an algebraic group~$G$,
many questions about~$X(k)$
can be reformulated in terms of Galois cohomology;
usually, non-abelian Galois cohomology
(cf.~\cite[\textsection5]{serrecg}, \cite{giraud}),
unless~$G$ is commutative.
All of the tools of Galois cohomology can then be brought to bear on the problems discussed here.
They are key to the proofs of the results we review below.

\subsubsection{Existence and density of rational points}
\label{subsubsec:existencegalois}

The following theorem, which is due to Borovoi,
builds on
earlier work of
Voskresenski\u{\i} and Sansuc
dealing with
the case of principal
homogeneous spaces (see~\cite{sansuclin}).

\begin{thm}[\cite{borovoi}]
\label{th:borovoi}
Let~$X$ be a smooth and proper variety over a number field~$k$.
Assume that~$X$ is birationally equivalent
to a homogeneous space of a connected linear algebraic group
and that the stabiliser of a geometric point of this homogeneous space is connected.
Then~$X(k)$ is dense in $X(\A_k)^{\Br(X)}$.
\end{thm}

The
hypothesis that the algebraic group be linear can be removed,
at the expense of replacing $X(\A_k)^{\Br(X)}$ with
 $X(\A_k)^{\Br(X)}_\poonen$ and assuming that the
Tate--Shafarevich group of the underlying abelian variety
is finite (see~\cite[Th.~A.1]{boctsko}).

The hypothesis that the stabiliser of a geometric point be connected, however, is more serious.
Whether one can dispense with it is an open problem, and it is one of the main challenges posed by Conjecture~\ref{conj:ctrc}.
A positive solution must lie deep as it would solve the inverse Galois problem as well
(see~\cite[\textsection3.5]{serretopics}).

Let~$X$ be a smooth and proper variety over a number field~$k$, birationally equivalent to a homogeneous space of a connected linear algebraic group~$G$.
Let $H_{\bar x} \subseteq G \otimes_k \bark$ denote the stabiliser of a $\bark$\nobreakdash-point~$\bar x$
of the homogeneous space in question.
It is known that the task of proving the density of~$X(k)$ in~$X(\A_k)^{\Br(X)}$ can be reduced to the case
where $G=\mathrm{SL}_n$ and~$H_{\bar x}$ is finite
(see~\cite{lucchinireduction},
\cite{demarchelucchinireduction}).
It is also known that~$X(k)$ is dense in~$X(\A_k)^{\Br(X)}$
if~$G$ is semi-simple and simply connected
and~$H_{\bar x}$ is finite and abelian
(see~\cite{borovoi}).
Some progress
on the case of a finite solvable group~$H_{\bar x}$
is achieved in~\cite{harariquelques} and in~\cite{lucchiniapprox}.

As it turns out, the proofs of the two positive results
from~\cite{borovoi} that
we have just quoted,
in which either $H_{\bar x}$ is connected or~$G$ is semi-simple and simply connected and~$H_{\bar x}$
is finite and abelian, only make use of the algebraic subgroup~$\Br_1(X)$
of~$\Br(X)$
and hence imply, as a by-product, that $X(\A_k)^{\Br(X)}=X(\A_k)^{\Br_1(X)}$.
In fact, in both cases, the equality $\Br(X)=\Br_1(X)$ holds.
When~$H_{\bar x}$ is connected,
this is a theorem of Borovoi, Demarche and Harari~\cite{borovoidemarcheharari},
which goes back to Bogomolov~\cite{bogomolov} when~$G$ is semi-simple and
simply connected; see also~\cite{borovoibis}.
When~$G$ is semi-simple and simply connected
and~$H_{\bar x}$ is finite and abelian,
combine~\cite[Appendix~A]{lucchinitransform}, \cite[Ex.~4.10]{ctsanmumbai}
and the theorem of Fischer~\cite{fischer}
according to which~$X$ is geometrically rational if $G=\mathrm{GL}_n$
and~$H_{\bar x}$ is finite and abelian.

On the other hand,
in general,
transcendental elements of~$\Br(X)$ do exist and cannot be ignored:
Demarche, Lucchini Arteche and Neftin~\cite{demarchelucchinineftin}
recently gave an example showing
that even when $G=\mathrm{SL}_n$ and~$H_{\bar x}$ is a semi-direct product of
two finite abelian groups,
it can happen that
$X(\A_k)^{\Br(X)} \neq X(\A_k)^{\Br_1(X)} = X(\A_k)$.
Thus,
the obstacles to extending Theorem~\ref{th:borovoi} to arbitrarily disconnected stabilisers
seem closely related to the usual difficulties inherent to
transcendental elements of the Brauer group
(see Remark~\ref{rmks:conjct}~(vii)
and~\textsection\ref{subsubsec:k3surfaces}).

\subsubsection{$R$\nobreakdash-equivalence}

We refer the reader to~\cite[Ch.~6]{voskbook}
and~\cite{ctflasque}
for a thorough discussion of $R$\nobreakdash-equivalence on linear algebraic groups.
With the notation
of~\textsection\ref{subsubsec:existencegalois}, descent theory provides a complete description of the $R$\nobreakdash-equivalence classes on~$X(k)$ when~$G$ is an algebraic torus (see~\cite{ctsantores}
or~\textsection\ref{subsec:descent} below).
When $H_{\bar x}=1$,
descriptions of $X(k)/R$ are available
for more general linear algebraic groups~$G$;
for example,
semi-simple adjoint groups of classical type
are dealt with in~\cite{merkadjoint}.
Without assumptions on~$G$, the set $X(k)/R$ is known to be finite when $H_{\bar x}=1$
(see~\cite{gille} and~\cite[\textsection1.3]{gillespecialisation}).
Finiteness of~$X(k)/R$ for homogeneous spaces with $H_{\bar x}\neq 1$
remains a challenge, even when~$H_{\bar x}$ is connected.

\subsubsection{Zero-cycles}

Thanks to the theorem of Liang~\cite{liangarithmetic} stated
at the end of~\textsection\ref{subsec:zerocycles},
the results of Borovoi~\cite{borovoi} about rational points imply their analogues
for zero-cycles.  Thus, in
the notation of~\textsection\ref{subsubsec:existencegalois}
and of~\textsection\ref{subsec:zerocycles},
Conjectures~\ref{conj:e1} and~(E)
hold for~$X$ as soon as
either
$H_{\bar x}$ is connected
or~$G$ is semi-simple and simply connected
and~$H_{\bar x}$ is finite and abelian.
For more general stabilisers, these conjectures are open.

\subsubsection{Integral points}
\label{intpoints:galois}

The following theorem is due to Borovoi and Demarche~\cite{borovoidemarche}.
It builds on previous work of
Colliot-Thélène and Xu~\cite{ctxu}
and of Harari~\cite{hararidefautforte}.

\begin{thm}
Let~$X$ be a homogeneous space of a connected linear algebraic group~$G$
over a number field~$k$.
Assume that the stabiliser of a geometric point is connected.
Let~$S \subset \Omega$ be a finite subset containing~$\Omega_\infty$
and $k_S=\prod_{v\in S}k_v$.
If~$F(k_S)$ is non-compact for every $k$\nobreakdash-simple factor~$F$ of the
semi-simple simply connected part of~$G$,
then~$X(k)$ is dense in $X(\A_k)^{\Br(X)}$ off~$S$
(see Definition~\ref{def:densityoff}).
\label{th:borovoidemarche}
\end{thm}

We note that the non-compactness hypothesis is vacuous unless~$k$ is totally real.

Theorem~\ref{th:borovoidemarche} is a classical result of
Kneser and Platonov when~$G$ is semi-simple and simply connected
and~$X=G$; in this case, the non-compactness hypothesis
is also necessary for the conclusion to hold.

The quotient $\Br(X)/\Br_0(X)$ is finite in Theorem~\ref{th:borovoi}
(see Remark~\ref{rmks:conjct}~(ii)),
but in Theorem~\ref{th:borovoidemarche}, it need not be
(example: $X=\Gm$).
In the special case of torsors under tori,
refinements of Theorem~\ref{th:borovoidemarche} involving only a finite subgroup
of $\Br(X)/\Br_0(X)$, leading to concrete criteria for the existence of integral points,
have been worked out by Wei and Xu, see~\cite{weixu}, \cite{weixumult}.

\subsection{The method of descent}
\label{subsec:descent}

It has been understood since Fermat that in order to answer questions about~$X(k)$,
it is sometimes both useful and easier
to study~$Y(k)$ for certain
auxiliary varieties~$Y$ equipped with a morphism $f:Y\to X$.
When~$X$ is a smooth and proper curve of genus~$1$,
the classical theory of descent considers morphisms~$f$ which are étale covers
and varieties~$Y$ which are again curves of genus~$1$ (see~\cite[Ch.~X]{silvermantopics}).
A similar approach, with~$f$ a torsor under an algebraic torus, was pioneered by Colliot-Thélène and Sansuc
to investigate the rational points of geometrically rational varieties (see~\cite{ctsandescent1}, \cite{ctsandescent2}).
Subsequent improvements have led to a common generalisation of these two types of descent,
in which~$Y$ is also allowed to be a torsor under an arbitrary (possibly disconnected) linear algebraic group
(see~\cite{skorobeyond}, \cite{harariweakapprox}, \cite{skobook}, \cite{harskononab}, \cite{hararinonabelian}).
In this section, we content ourselves with summarising the method 
in the simplest possible setting,
as it is described in~\cite{ctsandescent2};
this will be sufficient for the applications given in~\textsection\ref{subsubsec:examples}.

\subsubsection{Definitions and main statements}
\label{subsubsec:defmainstatements}

Let us fix a field~$k$, a separable closure~$\bark$
and a smooth, separated and geometrically irreducible variety~$X$ over~$k$.
We denote by $\bark[X]$ the ring of invertible functions on $X\otimes_k \bark$
and assume that $\bark[X]^*=\bark^*$.  This assumption is satisfied if~$X$ is proper,
which is the case of main interest, though it is sometimes useful to allow
non-proper~$X$ as well (see Remark~\ref{rk:descentandopen}
and~\textsection\ref{subsubsec:examples} below).
We also assume, for simplicity, that $\Pic(X \otimes_k \bark)$ is a free $\Z$\nobreakdash-module
of finite rank.
This is so, in particular, when~$X$ is rationally connected.

The Galois group $G_k=\Gal(\bark/k)$
acts on $\hat T = \Pic(X \otimes_k \bark)$.
Let $T=\Homrond(\hat T,\Gm)$ denote the corresponding algebraic torus over~$k$.
(An algebraic torus is an algebraic group over~$k$ which over~$\bark$ becomes
isomorphic to a product of copies of~$\Gm$.)
Recall that a \emph{torsor under~$T$, over~$X$}, is a variety~$Y$ over~$k$
endowed with a surjective morphism $f:Y\to X$
and with an action of~$T$ which preserves the geometric fibres of~$f$ and is simply transitive on each of them.
Up to isomorphism, torsors under~$T$, over~$X$, are classified by the étale cohomology group $H^1(X,T)$.
The latter fits into the exact sequence of low degree terms of the Hochschild--Serre spectral sequence,
which, in view of the equality $\bark[X]^*=\bark^*$, takes the following form:
\begin{align}
\label{eq:univt}
\xymatrix@C=1em@R=2ex{
0 \ar[r] & H^1(k,T) \ar[r] & H^1(X,T) \ar[r] & H^1(X \otimes_k \bark, T)^{G_k} \ar[r] &
H^2(k,T) \ar[r] & H^2(X,T)\rlap{\text{.}}
}
\end{align}
As $\hat T \simeq \Z^N$ and $H^1(X\otimes_k \bark,\Gm)=\Pic(X\otimes_k \bark)$,
there are canonical isomorphisms
\begin{align*}
H^1(X \otimes_k \bark,\Homrond(\hat T,\Gm))=\Hom(\hat T,H^1(X \otimes_k \bark,\Gm))=\End(\hat T)\rlap{\text{,}}
\end{align*}
through which the identity of~$\hat T$ gives rise to a canonical element
$\alpha \in H^1(X \otimes_k \bark, T)^{G_k}$.  Any torsor under~$T$, over~$X$, whose isomorphism class over~$\bark$
is~$\alpha$ is called a \emph{universal torsor} over~$X$
(see~\cite[\textsection2.0]{ctsandescent2}).

\begin{prop}
\label{prop:decomp}
Letting $f:Y\to X$ run over the
 isomorphism classes of universal torsors over~$X$,
there is a canonical decomposition of~$X(k)$ as a disjoint union:
\begin{align}
\label{eq:decomp}
X(k) = \coprod_{f:Y\to X} f(Y(k))\rlap{\text{.}}
\end{align}
In addition, if~$X$ is proper, every $R$\nobreakdash-equivalence class of~$X(k)$ is contained
in $f(Y(k))$ for some universal torsor $f:Y\to X$.
\end{prop}

\begin{proof}
If $X(k)=\emptyset$, this is clear.  Otherwise, the choice of a rational point
determines a retraction of the natural map
$H^2(k,T)\to H^2(X,T)$, which is therefore injective;
by~\eqref{eq:univt}, it follows that~$\alpha$ can be lifted to some $\beta_0 \in H^1(X,T)$.
Let $x \in X(k)$.
We still need to check
that there exists a unique universal torsor $f:Y\to X$
such that $x\in f(Y(k))$, or equivalently, that there exists a unique lifting $\beta \in H^1(X,T)$
of~$\alpha$ whose pull-back $\beta(x) \in H^1(k,T)$ along $x\in X(k)$ vanishes.
It is clear from~\eqref{eq:univt} that $\beta=\beta_0-\beta_0(x)$ is the unique such lifting.
Finally, the assertion about $R$\nobreakdash-equivalence follows from the remark that
$H^1(k,T)=H^1(\A^1_k,T)$, which implies that $\beta(x)=\beta(y)$ if $x,y\in X(k)$
are
directly $R$\nobreakdash-equivalent and~$X$ is proper.
\end{proof}

As a consequence of~\eqref{eq:decomp},
if~$k$ is a number field, then
\begin{align}
\label{eq:seqinclusions2}
X(k)\subseteq \bigcup_{f:Y\to X} f(Y(\A_k)) \subseteq X(\A_k)\rlap{\text{,}}
\end{align}
where $f:Y\to X$ ranges over the universal torsors over~$X$.
The analogy
with~\eqref{eq:seqinclusions} calls out for a comparison theorem.

\begin{thm}[Colliot-Thélène and Sansuc~\cite{ctsandescent2}]
\label{th:maindescent}
Let~$X$ be a smooth, separated and geometrically irreducible variety over
a number field~$k$,
such that $\bark[X]^*=\bark^*$ and that
$\Pic(X \otimes_k \bark)$ is a free $\Z$\nobreakdash-module
of finite rank.
Then
\begin{align}
\label{eq:decompdescent}
X(\A_k)^{\Br_1(X)}=\bigcup_{f:Y\to X} f(Y(\A_k))\rlap{\text{,}}
\end{align}
where $f:Y\to X$ ranges over the isomorphism classes of universal torsors over~$X$.
In addition, if~$X$ is proper, there are only finitely many $f:Y\to X$
with $Y(\A_k)\neq\emptyset$.
\end{thm}

We recall that
 $\Br_1(X)=\Ker(\Br(X)\to\Br(X\otimes_k\bark))$
(see Remark~\ref{rmks:conjct}~(vii)).
The proof of Theorem~\ref{th:maindescent} rests on the Poitou--Tate duality theorem relating
the Galois cohomology of~$T$ with that of~$\hat T$, and on~\eqref{eq:univt};
see~\cite[\textsection6.1]{skobook}.

The next corollary shows how Proposition~\ref{prop:decomp}
and Theorem~\ref{th:maindescent} effectively reduce certain questions on~$X$ to similar questions on
the universal torsors over~$X$.
In its statement, we denote by~$X'$ (resp.~$Y'$) any smooth and proper variety, over~$k$, which
contains~$X$
(resp.~$Y$) as a dense open subset.

\begin{cor}
\label{cor:maindescent}
Let~$X$ be a smooth, separated and geometrically irreducible variety over a number field~$k$,
such that
$\bark[X]^*=\bark^*$ and $\Pic(X\otimes_k \bark)$ is torsion-free.
\begin{enumerate}
\itemsep 4pt
\item If $Y'(k)$ is dense in $Y'(\A_k)$ for every universal torsor $f:Y\to X$,
then $X'(k)$ is dense in $X'(\A_k)^{\Br_1(X')}$.
\item Suppose that~$X$ is proper and that
for every universal torsor $f:Y\to X$, all rational points of~$Y'$ are $R$\nobreakdash-equivalent.
Then the set $X(k)/R$ is finite and the decomposition~\eqref{eq:decomp}
is the partition of~$X(k)$ into its $R$\nobreakdash-equivalence classes.
\end{enumerate}
\end{cor}

When~$X$ is proper,
Corollary~\ref{cor:maindescent}
is an immediate consequence
of Proposition~\ref{prop:decomp}
and Theorem~\ref{th:maindescent}.  For non-proper~$X$, one also needs a theorem of Harari
according to which $X(\A_k)^{\Br_1(X)}$
is dense in $X'(\A_k)^{\Br_1(X')}$
(see~\cite[Th.~1.4]{ctbudapest}).

According to a conjecture of Colliot-Thélène and Sansuc,
the set~$Y'(k)$ should be dense in $Y'(\A_k)$ for every universal torsor $f:Y\to X$
as soon as~$X$ is proper and geometrically rational (see~\cite[\textsection4]{cttoulouse}).
We note that assuming~$X$ to be proper and rationally connected would not be enough, here,
since $X(\A_k)^{\Br(X)}$ can prevent~$X(k)$ from being
dense in $X(\A_k)^{\Br_1(X)}$ for such varieties; see
Remark~\ref{rmks:conjct}~(vii).
It is an open question whether the rational points of~$Y'$ are all $R$\nobreakdash-equivalent
whenever~$Y$ is a universal torsor over a smooth, proper and  geometrically rational surface.
For universal torsors over smooth, proper and geometrically rational threefolds, this can fail
(see~\cite[\textsection5]{ctsanpremieredescente}).

\begin{rmk}
\label{rk:descentandopen}
Given~$X$ and~$k$ as
in Theorem~\ref{th:maindescent},
one can apply
descent theory to~$X$ as we did above, but one can also apply it to~$X'$.
By Corollary~\ref{cor:maindescent}~(1),
we conclude that~$X'(k)$ is dense in $X'(\A_k)^{\Br_1(X')}$
as soon as
 weak approximation holds
 \emph{either}
for every universal torsor over~$X$ 
 \emph{or}
for every universal torsor over~$X'$.
This raises the question of the relation between these two conditions.
In fact, it is possible to prove that if every universal torsor
over~$X$ satisfies weak approximation, then so does every universal torsor over~$X'$.
In other words,
as far as rational points are concerned,
one might as well resort to proper descent (descent theory for proper varieties, as it was originally envisaged)
in any situation
covered by
 its more recent counterpart for varieties
subject only
to the assumption $\bark[X]^*=\bark^*$.
What, then, is the point of allowing a non-proper~$X$?
For descent theory to be applicable, one needs a way to identify the universal
torsors as varieties, at least up to a birational equivalence.
This requires, among others,
a detailed knowledge of~$\Pic(X\otimes_k \bark)$ as a Galois module (together with explicit generators, most of the time), see~\textsection\ref{subsubsec:localdescription} below.
In many situations, the variety~$X$ is given and a smooth compactification $X \subseteq X'$ is only known to exist by Hironaka's theorem; in such cases $\Pic(X'\otimes_k \bark)$ can be much more difficult to describe than $\Pic(X\otimes_k \bark)$.
Descent theory for open varieties was used
for the first time in~\cite{ctskodescent}.
\end{rmk}

\subsubsection{Local description of the universal torsors}
\label{subsubsec:localdescription}

Let~$k$ and~$X$ be as at the beginning of~\textsection\ref{subsubsec:defmainstatements}.
Let $U\subset X$ be a dense open subset, small enough that $\Pic(U \otimes_k \bark)=0$.
There are a torsor $T^1_U$ under an algebraic torus $T^0_U$ over~$k$
and a morphism $u:U \to T^1_U$ through which any morphism from~$U$
to a torsor under a torus over~$k$ factors uniquely.
Concretely, if we let  $\hat T_U^0=\bark[U]^*/\bark^*$,
then $T_U^0=\Homrond(\hat T_U^0,\Gm)$
and $T_U^1(\bark)$
is the set of retractions
of the projection $\bark[U]^* \to \bark[U]^*/\bark^*$
(any two such retractions indeed differ by a unique homomorphism $\hat T_U^0\to \bark^*$),
see~\cite[Lemma~2.4.4]{skobook}.
Let us denote by $\hat Q \subset \Div(X \otimes_k \bark)$ the subgroup of divisors
supported on $(X \setminus U)\otimes_k\bark$.
The exact sequence of torsion-free Galois modules
\begin{align}
\xymatrix{
0 \ar[r] & \hat T_U^0 \ar[r] & \hat Q \ar[r] & \hat T \ar[r] & 0
}
\end{align}
 gives rise, dually, to an exact sequence of algebraic tori
\begin{align}
\xymatrix{
1 \ar[r] & T \ar[r] & Q \ar[r] & T^0_U \ar[r] & 1\rlap{\text{.}}
}
\end{align}
We can view~$Q$ as a torsor under~$T$, over~$T^0_U$.

\begin{thm}[Colliot-Thélène and Sansuc~\cite{ctsandescent2}]
\label{th:localdesc}
Let $f:Y\to X$ be a universal torsor.
There exists $b \in T^1_U(k)$ such that $f^{-1}(U)\to U$
is isomorphic, as a torsor,
to the pull-back of $Q\to T^0_U$
by the morphism $U \to T^1_U \simeq T^0_U$
obtained by composing~$u$ with the trivialisation of~$T^1_U$ determined by~$b$.
\end{thm}

\subsubsection{Examples}
\label{subsubsec:examples}

The simplest examples in which descent is successful are smooth compactifications of torsors under tori and certain conic bundle surfaces.

\newcommand{\citectsextore}{\cite[\textsection III]{ctsandescent1}}
\begin{example}[\citectsextore]
Let~$X$ be a smooth, proper variety over a number field~$k$, containing a torsor under a torus as a dense
open subset.  Let $U \subset X$ be this open subset.
By its very definition, the map $u:U\to T^1_U$ is an isomorphism.
On the other hand, the $\Z$\nobreakdash-module~$\hat Q$ possesses a basis stable under the action of~$G_k$,
hence~$Q$ is an open subset of an affine space.
We conclude, thanks to Theorem~\ref{th:localdesc},
that universal torsors over~$X$ are birationally equivalent, over~$k$, to a projective space.
By Corollary~\ref{cor:maindescent}, it follows that $X(k)$ is dense in $X(\A_k)^{\Br(X)}$ and
that $X(k)/R$ is finite.
\end{example}

\newcommand{\citectsextorebis}{open variant of \cite[\textsection IV]{ctsandescent1}, \cite[\textsection1]{bms}}
\begin{example}[\citectsextorebis]
\label{ex:univtorsorconicbundle}
Let~$L/k$ be a separable quadratic extension.
Let $P(t)=c\prod_{i=1}^d (t-a_i) \in k[t]$
be non-constant, with pairwise distinct $a_1,\dots,a_d \in k$.
Let~$X$ be the smooth affine surface
defined by the equation $N_{L/k}(x)=P(t)$
and let $U \subset X$ be the open subset defined by $\prod_{i=1}^{d-1}(t-a_i)\neq 0$.
One readily checks
that $\bark[X]^*=\bark^*$, that $\Pic(X\otimes_k\bark)$ is torsion-free,
that $\Pic(U\otimes_k\bark)=0$,
that $(t-a_i)_{1\leq i\leq d-1}$ forms a $\Z$\nobreakdash-basis of $\bark[U]^*/\bark^*$
and that the morphisms $U \to T^1_U \simeq T^0_U$ appearing in Theorem~\ref{th:localdesc}
are the morphisms $U \to \Gm^{d-1}$, $(x,y,t)\mapsto (b_i(t-a_i))_{1\leq i\leq d-1}$
for $b_1,\dots,b_{d-1}\in k^*$.
It is then easy to deduce from Theorem~\ref{th:localdesc} that any universal torsor
over~$X$ is birationally equivalent to the affine variety defined by the system
\begin{align}
\label{eq:univtorsor}
N_{L/k}(x_i)=b_i(t-a_i)\qquad (1 \leq i \leq d)
\end{align}
for some $b_1,\dots,b_d\in k^*$ such that $\prod_{i=1}^d b_i=c$.
By a stroke of luck, this is precisely a variety whose arithmetic
is accounted for, when $k=\Q$, by additive combinatorics:
Example~\ref{ex:bms}
 applies
(take $s=2$, $r=d+1$, $L_1=\dots=L_r=L$,
 $f_i=b_i(u_1-a_iu_2)$ for $i\in\{1,\dots,d\}$
and $f_{d+1}=u_2$).  We conclude, by Corollary~\ref{cor:maindescent},
that if $k=\Q$ and~$X'$ denotes a smooth and proper model of~$X$,
then~$X'(\Q)$ is dense in $X'(\A_\Q)^{\Br(X')}$.
\end{example}

This is a striking example of a fruitful combination of the descent method with an analytic
method, namely additive combinatorics.
Smooth and proper models of the varieties 
dealt with in
Example~\ref{ex:univtorsorconicbundle}
are conic bundle surfaces over~$\P^1_\Q$ with~$d$, if~$d$ is even, or~$d+1$, if~$d$ is odd,
singular geometric fibres.  It is remarkable that~$d$ is allowed to be arbitrarily
large: until the paper~\cite{bms}, even
the Zariski density of~$X(\Q)$ in~$X$ was not known for a single minimal conic bundle surface over~$\P^1_\Q$ with more than~$6$
singular geometric fibres
(see Remark~\ref{rmks:conjct}~(iii)).

Example~\ref{ex:univtorsorconicbundle} can be extended in many directions.

Browning, Matthiesen and Skorobogatov~\cite{bms} prove
that for any conic bundle surface $X \to \P^1_\Q$ whose
singular fibres lie over rational points of~$\P^1_\Q$,
with~$X$ smooth and proper over~$\Q$,
the set~$X(\Q)$
is dense in $X(\A_\Q)^{\Br(X)}$.  The argument
is exactly the same as in Example~\ref{ex:univtorsorconicbundle}, with a reduction to Example~\ref{ex:bms} too,
the only difference being that analysing the universal torsors requires a bit more care.

Eliminating~$t$ from~\eqref{eq:univtorsor}
leads to the intersection of $d-1$ quadrics given by
 $N_{L/k}(x_1)/b_1+a_1=\dots=N_{L/k}(x_d)/b_d+a_d$.
The surfaces of Example~\ref{ex:univtorsorconicbundle}
with~$d=3$
were first explored by Châtelet (see~\cite{chatelet}).
A simultaneous study of
these surfaces
(and, more generally, of the surfaces of Example~\ref{ex:univtorsorconicbundle}
with~$P(t)$ any separable polynomial of degree~$3$ or~$4$,
smooth and proper models of which are
now called Châtelet surfaces)
and of intersections of two quadrics
was carried out by Colliot-Thélène, Sansuc and Swinnerton-Dyer~\cite{cssI}, \cite{cssII},
who
proved, using Corollary~\ref{cor:maindescent},
 that for any Châtelet surface~$X$ over a number field~$k$,
the set~$X(k)$ is dense in $X(\A_k)^{\Br(X)}$, the set $X(k)/R$
is finite and the decomposition~\eqref{eq:decomp}
is the partition of~$X(k)$ into its $R$\nobreakdash-equivalence classes
(which leads to an explicit description of the $R$\nobreakdash-equivalence classes;
see \cite[Th.~8.8]{cssII} and \cite[p.~446]{ctsansuc} for some examples).
To ensure that the universal torsors satisfy the hypotheses
of Corollary~\ref{cor:maindescent},
they applied the fibration
 method, in the form of Proposition~\ref{prop:babycase} below;
this is an instance of a combination of the descent and fibration methods.
We note that the surface of Example~\ref{ex:isko} is a Châtelet surface.
It is still unknown whether the rational points of smooth and proper models
of the surfaces
of Example~\ref{ex:univtorsorconicbundle}
fall into finitely many
 $R$\nobreakdash-equivalence classes
when $d\geq 5$,
even if~$k=\Q$.

One can also replace, in Example~\ref{ex:univtorsorconicbundle},
the quadratic extension~$L/k$ with an arbitrary finite separable extension.
In general, no explicit smooth compactification of the resulting smooth affine variety~$X$ is known
(an exception being
the case when~$L/k$ is cyclic and either~$d$ or~$d+1$ is divisible by~$[L:k]$,
see~\cite{vavcompactifications});
thus, open descent, possibly on a partial compactification of~$X$, seems the only way out.
As a consequence, nothing is known on $R$\nobreakdash-equivalence when $[L:k]\geq 4$ and $d\geq 2$ (recall that Corollary~\ref{cor:maindescent}~(2) requires properness).

As was first noticed by Salberger,
the variety defined by the system~\eqref{eq:univtorsor} falls within the range of
applicability of the circle method,
for any finite extension $L/k$ of number fields, when $d=2$
(see~\cite[p.~331]{cttoulouse}).
By combining an open descent with the circle method,
Colliot-Thélène, Harari, Heath-Brown, Schindler, Skorobogatov and Swarbrick Jones
were thus able to prove that for any finite extension $L/k$ of number fields, 
any $m,n\in \N$
and any $c \in k^*$,
if~$X$ denotes a smooth and proper model of the affine variety defined by the equation $N_{L/k}(x)=ct^m (t-1)^n$,
then~$X(k)$ is dense in $X(\A_k)^{\Br(X)}$
(see~\cite{hbsko},
\cite{ctbudapest}, \cite{cthasko},
\cite{swarbrickjones}, \cite{schindlerskorobogatov}).
Nevertheless, even when~$L/k$ is biquadratic,
the absence of an explicit smooth and proper model makes it difficult
to determine whether $X(\A_k)^{\Br(X)}\neq\emptyset$
(see~\cite[\textsection2]{cthasko}, \cite{weinormtori}).

Derenthal, Smeets and Wei~\cite{dsw} show that
for any number field~$L$ and any quadratic polynomial $P(t)\in\Q[t]$,
the set~$X(\Q)$ is dense in $X(\A_\Q)^{\Br(X)}$
if~$X$ denotes a smooth and proper model
of the affine variety defined by $N_{L/\Q}(x)=P(t)$.
Their proof consists in applying Theorem~\ref{th:localdesc} and
Corollary~\ref{cor:maindescent} to this affine variety:
the universal torsors
turn out to be
precisely the varieties studied by Browning and Heath-Brown~\cite{bhbquadratic} by sieve methods (using ideas from the circle method).

Finally,
by improving
the additive
combinatorics arguments of~\cite{bms},
Browning and Matthiesen~\cite{browningmatthiesen}
were able to extend the assertion
of Example~\ref{ex:bms} from quadratic to arbitrary number fields~$L_i$.
Combining this with an open descent,
 they proved that
for any number field~$L$ and any $P(t) \in \Q[t]$
which splits as a product of linear polynomials over~$\Q$,
if~$X$ denotes a smooth and proper model of the affine variety defined by $N_{L/\Q}(x)=P(t)$,
then~$X(\Q)$ is dense in $X(\A_\Q)^{\Br(X)}$.
Using the same analytic input and a more refined open descent,
Skorobogatov~\cite{skorodescenttoric}
showed
that~$X(\Q)$ is dense in $X(\A_\Q)^{\Br(X)}$
whenever~$X$ is a smooth and proper variety endowed with a morphism to~$\P^1_\Q$ whose
singular fibres
lie over rational points of~$\P^1_\Q$
and whose generic fibre is birationally equivalent to a torsor under a constant torus
(\emph{i.e.}, a torus defined over~$\Q$).

\subsubsection{Integral points}
\label{intpoints:descent}

Theorem~\ref{th:maindescent} and Theorem~\ref{th:localdesc} were first applied to
study integral points on non-proper varieties
by Wei~\cite{weitorus},
who gave a positive answer to Question~\ref{q:codimdeux} for some varieties containing a torus
as a dense open subset.
Cao and Xu
 used descent to
show
that the conclusion of Theorem~\ref{th:borovoidemarche}
also holds for
 equivariant partial
compactifications of the homogeneous spaces which appear in its statement
(see~\cite{caoxugroupic}
and \cite{caoactiongroupelineaire},
which also relies on the fibration method).
Derenthal and Wei~\cite{derenthalwei}
applied
Theorem~\ref{th:maindescent} and Theorem~\ref{th:localdesc}
to analyse integral points on affine varieties defined by certain norm equations $N_{L/k}(x)=P(t_1,\dots,t_s)$.
Finally, descent theory was extended
 by Harari and Skorobogatov~\cite{harskodescentopen},
by Wei~\cite{weiopendescent},
and by Cao, Demarche and Xu~\cite{caoxumixing},
to smooth and geometrically irreducible varieties~$X$ such that
$\bark[X]^*\neq\bark^*$; this should prove useful for further applications
of descent to integral points.

\subsection{The fibration method}
\label{subsec:fibration}

We fix, once and for all, a smooth, proper and geometrically irreducible variety~$X$ over a number field~$k$.
The fibration method seeks to exploit the structure of a fibration $f:X\to Y$
to deduce information on the arithmetic of~$X$
from information on the arithmetic of~$Y$ and of many fibres of~$f$.
The reader will find in~\cite[\textsection2]{cttiruchirapalli} a detailed discussion of
the state of the art in~1998.

\subsubsection{Set-up}

The most basic situation in which such issues arise is that of a conic bundle surface $f:X\to \P^1_k$.
In this case, the smooth
fibres and the base satisfy weak approximation.  Nevertheless, as we have seen in Example~\ref{ex:isko}, it may happen
that $X(k)=\emptyset$ while $X(\A_k)\neq\emptyset$.
The weak approximation property is therefore not compatible with fibrations, which indicates that positive answers to the fibration
problem cannot be too simple.

Going back to an arbitrary~$X$,
one can still ask whether the more refined statements of
Conjecture~\ref{conj:ctrc}, Conjecture~\ref{conj:e1} or Conjecture~(E)
are compatible with fibrations.
(We leave  $R$\nobreakdash-equivalence aside, as nothing is known,
over number fields, about its behaviour with respect to fibrations.)
Let us assume, for simplicity, that $Y=\P^1_k$.
Multiple fibres of~$f$ and rational points of~$X$  are known to play against each other:
for example, as soon as~$f$ possesses at least~$5$ double geometric fibres,
the rational points of~$X$ are contained
in finitely many fibres (see~\cite[Cor.~2.4]{csspaucity};
this is now better understood in the light of Campana's theory of non-classical orbifolds, see~\cite[\textsection4]{campanafibres}).
We therefore assume that every fibre of~$f$ contains an irreducible component of multiplicity~$1$
(an assumption which holds,
in particular,
when the generic fibre of~$f$ is rationally connected; see~\cite{ghs}).
In this case,
letting
$X_c=f^{-1}(c)$,
it makes sense to ask:

\begin{question}
\label{q:fibre}
Under these hypotheses, is
\begin{align}
\label{eq:densesubset}
\bigcup_{c \in \P^1(k)} X_c(\A_k)^{\Br(X_c)}
\end{align}
a dense subset of $X(\A_k)^{\Br(X)}$?
\end{question}

The difficulty of this question increases with the \emph{rank} of~$f$,
defined as follows.

\begin{defn}[Skorobogatov~\cite{skorodescent}]
\label{def:rank}
A scheme of finite type over a field is \emph{split} if it contains a
geometrically
irreducible irreducible component of multiplicity~$1$.  
The set of points of~$\P^1_k$ above which the fiber of~$f$ is not split is a finite closed subset;
its degree over~$k$ (or, in other words, its geometric number of points)
is the \emph{rank} of~$f$.
\end{defn}

The following result,
implicitly stated in~\cite[p.~42]{cssI},
represents the ``baby case'' of the fibration method:
the rank is~$0$ and we ignore $\Br(X_c)$.
It is enough to deduce, for example, the Hasse--Minkowski theorem for quadric hypersurfaces of dimension~$d \geq 2$
from the case $d=2$, by considering pencils of hyperplane sections.

\begin{prop}
\label{prop:babycase}
Let $f:X \to \P^1_k$ be a morphism all of whose fibres are split.
The set
 $\bigcup_{c \in \P^1(k)} X_c(\A_k)$ is dense in $X(\A_k)$.
\end{prop}

\begin{proof}
Let us spread~$f$ out to a flat
${\sO_{k,S}}$\nobreakdash-morphism
$f':\sX\to \P^1_{\sO_{k,S}}$
for some finite subset $S \subset \Omega$
and fix
 $(P_v)_{v\in\Omega}\in X(\A_k)$.
After enlarging~$S$, we may assume
that~$S$ contains
the places at which we want to approximate~$(P_v)_{v \in \Omega}$,
that the fibres of~$f'$ are split,
and,
by the Lang--Weil--Nisnevich bounds on the number of rational
points of \emph{split} varieties over finite fields, that every closed fibre of~$f'$
contains a smooth rational point.  Such points can be lifted, by Hensel's
lemma, to $k_v$\nobreakdash-points of~$X_c$ for any $c \in \P^1(k)$ and any finite $v\notin S$.
Thus, the proposition will follow
from the existence of $c \in \P^1(k)$ and $(Q_v)_{v \in S} \in \prod_{v \in S} X_c(k_v)$
with~$Q_v$ arbitrarily close to~$P_v$ for each $v\in S$;
but the implicit function theorem provides such~$Q_v$ for any $c \in \P^1(k)$
which is close enough to $(f(P_v))_{v \in S} \in \prod_{v \in S} \P^1(k_v)$.
\end{proof}

\begin{rmk}
\label{rk:goeswrong}
What goes wrong in this proof in the presence of non-split fibres
is that the existence of smooth rational points in the closed fibres of~$f'$,
and hence of $k_v$\nobreakdash-points of~$X_c$ for $v \notin S$ and $c \in \P^1(k)$,
cannot be guaranteed.
Say, for example, that
for some $m \in \P^1(k)$,
the fibre~$X_m$ is integral but not geometrically irreducible,
so that the algebraic closure~$L_m$ of~$k$ in the function field of~$X_m$
is distinct from~$k$.
Suppose the Zariski closures of~$c$ and~$m$ in $\P^1_{\sO_{k,S}}$
meet in a closed point~$z$ which lies above $v \in \Spec(\sO_{k,S})$.
Without additional assumptions on~$f$,
there is then no way to ensure that $X_c(k_v)\neq\emptyset$ unless
the number field~$L_m$ happens to possess a place of degree~$1$ over~$v$.
If~$S$ was chosen large enough at the beginning of the argument,
the existence of such a place is equivalent to $f'^{-1}(z)$ being split.
\end{rmk}

The proof of Proposition~\ref{prop:babycase} ends with an application of weak approximation
on the projective line.
An easy variant based, instead, on the use of strong approximation off one place
on the affine line,
shows that the conclusion of Proposition~\ref{prop:babycase}
remains valid for fibrations of rank~$1$
(see~\cite{skofibrationmethod}).
Taking into account the
Brauer groups of the fibres is more difficult;
still in the case of rank at most~$1$, 
a very elaborate argument
allowed Harari~\cite{harariduke} to answer Question~\ref{q:fibre}
in the affirmative under the additional assumption that the geometric fibre of~$f$ is rationally connected.

The first result towards Question~\ref{q:fibre} for a fibration of rank~$2$ is due to Hasse.

\begin{prop}
\label{prop:hassecase}
Let $f:X \to \P^1_k$ be a morphism
whose fibres above $\P^1_k \setminus \{0,\infty\}$ are split
and whose generic fibre is a product of conics.  If $X(\A_k)\neq\emptyset$, then $X(k)\neq\emptyset$.
\end{prop}

\begin{proof}[Proof, after Hasse]
As above,
we spread~$f$ out,
for a large enough finite subset $S \subset \Omega$,
to a flat
${\sO_{k,S}}$\nobreakdash-morphism
$f':\sX\to \P^1_{\sO_{k,S}}$
whose split closed fibres contain a smooth rational point.
By the implicit function theorem, there exists $(P_v)_{v \in \Omega} \in X(\A_k)$
with $f(P_v)\notin \{0,\infty\}$ for all $v \in \Omega$.  Let $t_v \in k_v^*=\Gm(k_v)\subset \P^1(k_v)$
denote the coordinate of $f(P_v)$. By
Dirichlet's theorem on primes in arithmetic
progressions (if $k=\Q$, or otherwise by its generalisation in class field theory),
there exists $t_0 \in k^*$
arbitrarily close to~$t_v$ for finite $v \in S$,
with prescribed signs at the real places,
such that~$t_0$ is a unit outside~$S$ except at exactly one place,
say~$v_0$.
Let $c \in \P^1(k)$ be the point with coordinate $t_0 \in k^*=\Gm(k)\subset \P^1(k)$.
An appropriate choice of signs at the real places
together with
an application of Hensel's lemma (for $v\notin S$) and of the implicit function theorem
(for finite $v\in S$)
shows,
just as in the proof of Proposition~\ref{prop:babycase},
 that $X_c(k_v)\neq\emptyset$ for all $v\in\Omega\setminus\{v_0\}$.
Now, as~$X_c$ is a product of conics, this implies,
by quadratic reciprocity,
that $X_c(k_{v_0})\neq\emptyset$; and hence,
by the Hasse--Minkowski theorem for conics, that $X_c(k)\neq\emptyset$.
\end{proof}

The Hasse--Minkowski theorem for quadric surfaces $\sum_{i=0}^3 a_ix_i^2=0$
follows
by letting $f:X\to \P^1_k$ be the fibration into products of two conics
obtained by desingularising and compactifying the affine threefold defined by $a_0x_0^2+a_1x_1^2=-a_2x_2^2-a_3x_3^2=t$,
endowed with the projection to the~$t$ coordinate.

The arguments used in the proofs of Propositions~\ref{prop:babycase}
and~\ref{prop:hassecase}
and in Harari's paper~\cite{harariduke}
were extended in a series of works
by
Colliot-Thélène, Harpaz, Sansuc, Serre, Skorobogatov, Swinnerton-Dyer,
and the author
(see~\cite{ctsansucschinzel},
\cite{serrecollege},
\cite{sdpencils},
\cite{ctsd94},
\cite{ctsksd98},
\cite[\textsection\textsection7--9]{swdtopics},
\cite{hsw},
\cite[\textsection9]{hw}).
For fibrations of rank greater than~$2$,
Dirichlet's theorem on primes in arithmetic progressions is replaced
with a conjectural statement: either Schinzel's hypothesis,
a conjecture on the simultaneous prime values taken by a finite collection of
irreducible polynomials
(Dirichlet's theorem being the case of one linear polynomial),
or a presumably easier conjecture,
formulated in~\cite[\textsection9]{hw} and in~\textsection\ref{subsubsec:conj} below,
on the splitting behaviour, in fixed number fields,
of the primes dividing the values taken
by a finite collection of irreducible polynomials.
Schinzel's hypothesis is useful only under an abelianness assumption:
typically,
in the situation of Remark~\ref{rk:goeswrong}, the extension~$L_m/k$ should be abelian
for Schinzel's hypothesis to be applicable.
Abelianness is necessary to run a reciprocity argument as at the end of the proof of Proposition~\ref{prop:hassecase}.  By a new method
which avoids this reciprocity argument altogether, the abelianness assumption could be removed, thus leading to the following general theorem.

\newcommand{\citehw}{\cite[\textsection9]{hw}}
\begin{thm}[\citehw]
\label{th:mainfibration}
Let $f:X \to \P^1_k$ be a morphism whose generic fibre
is geometrically irreducible and rationally connected.
Assume Conjecture~\ref{conj:hw} (stated below).  Then
Question~\ref{q:fibre} has a positive answer. In particular, if the smooth
fibres of~$f$ satisfy Conjecture~\ref{conj:ctrc}, then so does~$X$.
\end{thm}

\subsubsection{A conjecture on strong approximation}
\label{subsubsec:conj}

Let $m_1,\dots,m_n \in \A^1_k$ be closed points.
For each $i \in \{1,\dots,n\}$,
fix a finite extension~$L_i$ of the residue field $k_i=k(m_i)$ of~$m_i$,
let $a_i \in k_i$ denote the coordinate of~$m_i \in \A^1_k$
and fix $b_i\in k_i^*$.
Let $F_i \subset R_{L_i/k}(\A^1_{L_i})$ denote the singular locus of the affine variety defined by $N_{L_i/k}(x_i)=0$
(see~\textsection\ref{subsubsec:sieve}); geometrically, this is a union of linear subspaces
of codimension~$2$.
Letting $\lambda,\mu$ denote the coordinates of~$\A^2_k$,
we can then consider the closed subvariety
$$W \subset
\left(\A^2_k \setminus \{(0,0)\}\right) \times \prod_{i=1}^n \left(R_{L_i/k}(\A^1_{L_i}) \setminus F_i\right)$$
defined by the equations $N_{L_i/k_i}(x_i)=b_i(\lambda-a_i\mu)$ for $i\in\{1,\dots,n\}$, together
with the smooth surjective map $p:W \to \P^1_k$ defined by $p(\lambda,\mu,x_1,\dots,x_n)=[\lambda:\mu]$.

\newcommand{\citehwconj}{\cite[\textsection9]{hw}}
\begin{conj}[\citehwconj]
\label{conj:hw}
Letting $W_c=p^{-1}(c)$,
the set $\bigcup_{c \in \P^1(k)} W_c(\A_k)$ is dense in~$W(\A_k)$.
\end{conj}

The reader will notice the immediate similarity between the definition of~$W$
and the equations of the universal torsors encountered in Example~\ref{ex:univtorsorconicbundle}.
What is crucial, here, is that~$W$ is not proper and~$W(\A_k)$ is endowed with the adelic topology,
introduced in~\textsection\ref{subsec:integralpoints};
Conjecture~\ref{conj:hw} is really a statement about
integral points.
In particular,
if~$v$ is a finite
place of~$k$ unramified in a finite extension~$L/k$
and~$\sF$ denotes the singular locus of the affine scheme defined by $N_{\sO_L/\sO_k}(x)=0$,
the existence of an $\sO_v$\nobreakdash-point of $R_{\sO_L/\sO_k}(\A^1_{\sO_L}) \setminus \sF$
whose norm from~$L$ to~$k$ has positive valuation can be shown to be equivalent to the existence
of a place of~$L$ of degree~$1$ above~$v$,
a property we already encountered in Remark~\ref{rk:goeswrong}.
A more down-to-earth formulation of Conjecture~\ref{conj:hw},
which makes clearer the analogy with Dirichlet's theorem (in the way we used it in the
proof of Proposition~\ref{prop:hassecase}),
can be found in \cite[\textsection9.1]{hw}.

It is likely that~$W$ always satisfies strong approximation off any given place.
One easily checks that this would imply Conjecture~\ref{conj:hw}
(see \cite[Cor.~9.10]{hw}).

There is also an obvious similarity between the statement of Conjecture~\ref{conj:hw} and Question~\ref{q:fibre}.  The point of Conjecture~\ref{conj:hw} is that the geometry of~$W$
can be much simpler than that of~$X$, so simple that a direct application of analytic or
algebraic methods
can solve particular cases of it.

\begin{example}
\label{ex:degree2}
If $\sum_{i=1}^n \deg(m_i)\leq 2$,
then~$W$ is the complement of a closed subset of codimension~$2$ in an affine
space; therefore it satisfies strong approximation off any given place (see the references
quoted after Question~\ref{q:codimdeux})
and hence Conjecture~\ref{conj:hw} holds in this case.
\end{example}

Combining Example~\ref{ex:degree2} with the proof of Theorem~\ref{th:mainfibration}
yields a positive answer to Question~\ref{q:fibre} for fibrations, into rationally connected varieties, of rank at most~$2$.
In particular, if the smooth
fibres of such a fibration satisfy Conjecture~\ref{conj:ctrc}, then so does~$X$
(a vast generalisation of Proposition~\ref{prop:hassecase}: the Brauer groups
of the smooth fibres are allowed to contain arbitrary non-constant classes).

Let us now illustrate the relevance of analytic methods to Conjecture~\ref{conj:hw}.
As a first example, the arguments which form the proof of Irving's result
mentioned in Example~\ref{ex:irving}, based on sieve methods, turn out
to also prove, in fact,
exactly the case of
Conjecture~\ref{conj:hw} needed to recover the main theorem of~\cite{irving}
as a consequence
of Theorem~\ref{th:mainfibration} and to extend it to general fibrations
sharing the same ``degeneration data'' (see~\cite[Th.~9.15 and~\textsection9.3]{hw}).
This is a case in which $\sum_{i=1}^n \deg(m_i)=4$.
The next theorem lists two results with a more significant scope.

\begin{thm}
\label{th:analyticcases}
If $k=\Q$,
Conjecture~\ref{conj:hw} holds in each of the following two cases:
\begin{enumerate}
\item $\deg(m_i)=1$ for all $i\in\{1,\dots,n\}$;
\item $\sum_{i=1}^n\deg(m_i)\leq 3$ and $\deg(m_i)=1$ for at least one $i\in\{1,\dots,n\}$.
\end{enumerate}
\end{thm}

Theorem~\ref{th:analyticcases}~(1) is due to Matthiesen~\cite{matthiesen};
it uses additive combinatorics and builds on the work of Browning and
Matthiesen~\cite{browningmatthiesen}.  Theorem~\ref{th:analyticcases}~(2)
is a very recent result of Browning and Schindler~\cite{browningschindler},
which, similarly, builds on the work of Browning and
Heath-Brown~\cite{bhb}.  The following corollary combines these analytic
results with the fibration method.

\begin{cor}
\label{cor:combinfibre}
Let $f:X \to \P^1_\Q$ be a morphism whose generic fibre is rationally connected.
Assume either that the non-split fibres of~$f$ lie over rational points of~$\P^1_\Q$,
or that~$f$
has a non-split fibre over a quadratic
point and has rank at most~$3$. Then Conjecture~\ref{conj:ctrc} holds for~$X$ as soon as it holds for the smooth
fibres of~$f$.
\end{cor}

One can further combine these results with
Theorem~\ref{th:borovoi}, obtained by Galois cohomological methods,
and deduce that under the hypotheses of Corollary~\ref{cor:combinfibre}, the set~$X(\Q)$ is dense in $X(\A_\Q)^{\Br(X)}$
if the generic fibre of~$f$ is birationally equivalent
to a homogeneous space of a connected linear algebraic group
with connected geometric stabilisers.
The particular case of fibrations into torsors under a constant torus, over~$\Q$,
recovers and extends a number of the results mentioned in~\textsection\ref{subsubsec:examples},
notably those of~\cite{hbsko}, \cite{cthasko}, \cite{bms}, \cite{browningmatthiesen},
\cite{skorodescenttoric} and those of~\cite{dsw}.

\subsubsection{Zero-cycles and fibrations}
\label{subsubsec:zerocyclesandfibrations}

In spite of all of these various methods,
the density of~$X(k)$ in $X(\A_k)^{\Br(X)}$
is still unknown, unconditionally, for general conic bundle surfaces,
even when $k=\Q$.
The situation is better in the analogous context of zero-cycles:
for an arbitrary
number field~$k$,
Salberger~\cite{salberger} established 
Conjectures~\ref{conj:e1} and~(E) for conic bundle surfaces over~$\P^1_k$
 (see also~\cite{salbergerdescent}).
His proof was recast in terms of the fibration method, and then extended to more general
fibrations, by a number of authors (see~\cite{ctsd94},
\cite{ctsksd98},
\cite{ctreglees},
\cite{frossard},
\cite{vanhamel},
\cite{wittdmj},
\cite{liangastuce},
\cite{liangprojectif},
\cite{liangtowards},
\cite{smeets}).
All of these arguments rely on the number field version of
Dirichlet's theorem on primes in arithmetic progressions,
in the spirit of the proof of Proposition~\ref{prop:hassecase}.
When applied to zero-cycles,
the ideas which underlie the proof of Theorem~\ref{th:mainfibration}
yield the following result,  based on an unconditional
version of Conjecture~\ref{conj:hw} for zero-cycles
instead of Dirichlet's theorem.

\begin{thm}[\cite{hw}]
\label{th:mainfibrationzerocycles}
Let $f:X \to Y$ be a morphism whose generic fibre is geometrically irreducible and rationally connected.
Suppose that~$Y$ is a projective space, or a curve
satisfying Conjecture~(E), or the product of a projective space with
such a curve.
If the smooth fibres of~$f$ satisfy Conjecture~\ref{conj:e1} or Conjecture~(E), then~$X$
satisfies the same conjecture.
\end{thm}

The hypotheses and conclusions of Theorem~\ref{th:mainfibrationzerocycles} being birational invariants,
an immediate induction on~$\dim(Y)$
reduces the proof of Theorem~\ref{th:mainfibrationzerocycles}
to the case in which~$Y$ is a curve.
The next step of the proof consists in translating the problem
into a question about \emph{effective} zero-cycles of degree~$d$ for some $d\gg 0$.
One then wants to approximate certain adelic points of $\Sym^d(X)$ by adelic points
lying in the fibres of the natural projection $\Sym^d(X)\to \Sym^d(Y)$ above rational points of~$\Sym^d(Y)$.
Let us assume, for simplicity, that $Y=\P^1_k$, so that $\Sym^d(Y)=\P^d_k$.
At this point, the same arguments as in the proof of Theorem~\ref{th:mainfibration}
lead one to the following generalisation of Conjecture~\ref{conj:hw}:
in the definition of~$W$,
replace~$\A^2$ with~$\A^{d+1}$,
and $b_i(\lambda-a_i\mu)$, for $i \in \{1,\dots,n\}$, with
$k_i$\nobreakdash-linear forms~$\phi_i$
subject to the assumption that the induced $k$\nobreakdash-linear
map $\prod \phi_i:k^{d+1}\to \prod_{i=1}^n k_i$ has maximal rank.
Now it turns out, exactly as in Example~\ref{ex:bms}, that
when the number of variables is large enough with respect
to the number of equations,
the geometry of~$W$ becomes very simple:
for $d\gg 0$, it is easy to see that~$W$ is the
complement of a closed subset of codimension~$2$ in an affine space.
In other words, the phenomenon encountered in Example~\ref{ex:degree2} can be forced to occur, in the context
of zero-cycles, by letting~$d$ grow. This is what makes Theorem~\ref{th:mainfibrationzerocycles} unconditional.

This sketch clearly shows that instead of assuming that the smooth fibres of~$f$
(above the closed, not necessarily rational, points of~$Y$)
satisfy Conjecture~\ref{conj:e1} or Conjecture~(E),
we might as well assume, instead, that they satisfy Conjecture~\ref{conj:ctrc}:
the proof of Theorem~\ref{th:mainfibrationzerocycles} would still yield
Conjectures~\ref{conj:e1} and~(E) for~$X$.
As was noticed by Liang~\cite{liangarithmetic},
applying this remark to the trivial fibration $V \times \P^1_k \to \P^1_k$
shows that a smooth, proper and rationally connected variety~$V$ over~$k$ satisfies Conjectures~\ref{conj:e1} and~(E)
if it satisfies Conjecture~\ref{conj:ctrc} over all finite extensions of~$k$.

Allowing more general bases~$Y$ than those in the statement of
Theorem~\ref{th:mainfibrationzerocycles} remains a challenge; we
refer to~\cite{liangrc} for some results in this direction.

Again, one can combine Theorem~\ref{th:mainfibrationzerocycles}
with Theorem~\ref{th:borovoi} and with the results on Conjecture~(E)
for curves recalled at the end of~\textsection\ref{subsec:zerocycles},
and deduce that Conjectures~\ref{conj:e1} and~(E) are true for the total space of an
arbitrary fibration,
over a projective space or over a curve whose Jacobian has a finite Tate--Shafarevich group,
into homogeneous spaces of connected linear algebraic groups
with connected geometric stabilisers.

\subsubsection{Pencils of abelian varieties}
\label{subsubsec:pencilsofabvar}

The reason why
Theorems~\ref{th:mainfibration} and~\ref{th:mainfibrationzerocycles} restrict to
fibrations with rationally connected generic fibre~$X_\eta$
is that this assumption, or, more generally, the assumption that~$X_\eta$
is simply connected and satisfies $H^2(X_\eta,\sO_{X_\eta})=0$,
makes it much easier to control
how the groups $\Br(X_c)/\Br_0(X_c)$
behave as~$c$ varies, and to relate them to the Brauer group of~$X$
(see~\cite{harariduke}, \cite{hararifleches}, \cite[\textsection4]{hw});
to start with, these quotients are then finite,
by Remark~\ref{rmks:conjct}~(ii).

A very elaborate technique for dealing with fibrations into principal homogeneous spaces
of abelian varieties was initiated by Swinnerton-Dyer~\cite{sdegloffstein}
and developed in~\cite{cssinv},
\cite{bendersd}, \cite{ctbender}, \cite{wittlnm}
for semi-stable families and in~\cite{sdcubic}, \cite{skosdkummer},
\cite{harpazskokummer} for certain families
with additive reduction.
Attempting to describe it here would lead us too far afield;
we refer the reader to the introductions of the above-cited papers for
more detail, and shall be content with mentioning two applications to Conjecture~\ref{conj:ctrc}.
Assuming the finiteness
of the Tate--Shafarevich groups of elliptic curves over quadratic fields,
Swinnerton-Dyer~\cite{sdcubic} showed
that $X(\A_\Q)\neq\emptyset$ implies $X(\Q)\neq\emptyset$
if $X \subset \P^4_\Q$ denotes the smooth cubic threefold defined
by $\sum a_ix_i^3=0$ for $a_0,\dots,a_4\in \Q^*$.
Assuming Schinzel's hypothesis and the finiteness
of the Tate--Shafarevich groups of elliptic curves over number fields,
we proved in~\cite{wittlnm}
that $X(\A_k)\neq\emptyset$ implies $X(k)\neq\emptyset$
for most del Pezzo surfaces of degree~$4$
and for all smooth intersections of two quadrics in~$\P^n_k$ for $n\geq 5$
(and therefore
they satisfy Conjecture~\ref{conj:ctrc}, in view of earlier work of Salberger and
Skorobogatov~\cite{salbergerskorobogatov}
based on the descent method).
This technique can even be applied to non-rationally connected varieties,
for example to diagonal quartic surfaces (see~\cite{sddiagonalquartic},
\cite[\textsection1.7.2]{wittlnm})
and to Kummer varieties (see~\cite{skosdkummer}, \cite{harpazskokummer}).

In a distinct but parallel direction,
Várilly-Alvarado~\cite{varillyisotrivial}
proved,
assuming the finiteness of the Tate--Shafarevich groups of elliptic curves,
that rational points of certain del Pezzo surfaces of degree~$1$
defined over~$\Q$ are dense in the Zariski topology,
by exploiting
the anti-canonical pencil of elliptic curves lying on such surfaces.

\subsubsection{Integral points}
\label{intpoints:fibration}

The fibration method has been used to study integral points, on the
total space of non-proper fibrations over the affine line,
by Colliot-Thélène, Harari and Xu~\cite{ctxu2}, \cite{cthararifamille}.
Difficulties similar to those outlined at the beginning
of~\textsection\ref{subsubsec:pencilsofabvar} quickly appear in this context too.
As a consequence, these two papers restrict to fibrations into
homogeneous spaces of certain semi-simple simply connected linear
algebraic groups and assume that all of the fibres are split.

Harpaz~\cite{harpazlogk3} recently transposed
 to the setting of integral points
 the technique initiated
by Swinnerton-Dyer
and discussed in~\textsection\ref{subsubsec:pencilsofabvar}
for dealing with pencils of principal homogeneous spaces of abelian varieties,
replacing abelian varieties with algebraic tori.
He obtains
unconditional results
in the case of certain
log\nobreakdash-$K3$ surfaces
fibred into affine conics over~$\P^1_\Q$.
These are the first unconditional existence results
for integral (or rational) points on log\nobreakdash-$K3$ (or~$K3$) surfaces
over number fields.

\bibliographystyle{myamsalpha}
\bibliography{slc}
\end{document}